\documentclass[12pt,a4paper]{article}
\usepackage{amsmath,amsfonts,amsthm,amssymb,latexsym}
\usepackage{xspace}
\usepackage{lipsum}
\usepackage{scrextend}
\usepackage{enumitem}
\usepackage{comment}
\usepackage{graphicx}

\newcommand{\Z}{{\mathbb Z}}

\newcommand{\emptyword}{\epsilon}

\def\disting#1{{\mathbf d}_{#1}}
\def\leftnbr#1{{\mathbf d}^L_{#1}}
\def\rightnbr#1{{\mathbf d}^R_{#1}}
\newcommand{\vv}{p}
\newcommand{\xx}{x} 
\newcommand{\ee}{E} 
\newcommand{\ii}{I}
\newcommand{\jj}{J}
\newcommand{\kk}{K}

\newcommand{\cR}{\mathcal{R}}

\newcommand{\fail}{{\bf{fail}}}
\renewcommand{\stop}{{\bf{stop}}}
\newcommand{\flip}{\textsf{flip}}
\newcommand{\true}{\texttt{true}}
\newcommand{\false}{\texttt{false}}
\def \p#1{{\rm pref}[#1]}
\def \s#1{{\rm suf}[#1]}
\def \f#1{{\rm f}[#1]}
\def \l#1{{\rm l}[#1]}

\newtheorem{theorem}{Theorem}[section] 
\newtheorem{lemma}[theorem]{Lemma} 
\newtheorem{proposition}[theorem]{Proposition}
\newtheorem*{proposition*}{Proposition}

\theoremstyle{definition}
\newtheorem{definition}[theorem]{Definition}
\newtheorem{example}[theorem]{Example}

\newtheorem{procedure}[theorem]{Procedure}

\newenvironment{mylist}{\begin{list}{}{
\setlength{\parskip}{0mm}
\setlength{\topsep}{2mm}
\setlength{\parsep}{0mm}
\setlength{\itemsep}{0.5mm}
\setlength{\labelwidth}{7mm}
\setlength{\labelsep}{3mm}
\setlength{\itemindent}{0mm}
\setlength{\leftmargin}{12mm}
\setlength{\listparindent}{6mm}
}}{\end{list}}
\newenvironment{myproclist}{\begin{list}{}{
\setlength{\parskip}{0mm}
\setlength{\topsep}{0mm}
\setlength{\parsep}{0mm}
\setlength{\itemsep}{2mm}
\setlength{\labelwidth}{0mm}
\setlength{\labelsep}{0mm}
\setlength{\itemindent}{0mm}
\setlength{\leftmargin}{4mm}
\setlength{\listparindent}{3mm}
}}{\end{list}}

\title{Artin groups of type $(2,3,n)$}

\author{Derek F. Holt and Sarah Rees}

\begin{document}
\maketitle

\begin{abstract}

Our main theorem is that the word problem in the
Artin group
$$G= \langle a,b,c \mid aba=bab, ac=ca, {}_{n}(b,c) = {}_{n}(c,b) \rangle,$$
for $n \geq 5$ can be solved
using a system $\cR$ of length preserving rewrite rules that, together
with free reduction, can be used to reduce any word over $\{a,b,c\}$ to a
geodesic word in $G$, in quadratic time.
	This result builds on work of Holt and Rees, and of Blasco, Cumplido and Morris-Wright, which proves the same result for all Artin groups that are either sufficiently large or 3-free.

Since every rank 3 Artin group is either spherical or in one of the categories 
covered by the previous results on which we build, it follows that any rank 3 Artin group has quadratic Dehn function. However we note that this and much more 
is a consequence of very recent work of Haettel and Huang; our contribution is to provide a particular kind or rewriting solution to the word problem for the
	non-spherical rank 3 Artin groups (and more).
\end{abstract}

\section{Introduction}
\label{sec:intro}

Our main result in this article  are the following.
\begin{theorem}
	\label{thm:main}
Let $G$ be the Artin group
$$G= \langle a,b,c \mid aba=bab, ac=ca, {}_{n}(b,c) = {}_{n}(c,b) \rangle,$$
for $n \geq 5$.
Then there is a system $\cR$ of length preserving rewrite rules that, together
with free reduction, can be used to reduce any word over $\{a,b,c\}$ to a
geodesic word in $G$, in quadratic time.

Furthermore, any two geodesic words representing the same element of $G$ are
related
by a sequence of rules in $\cR$, each of those involving just two generators. 
\end{theorem}


The notation ${}_{n}(x,y)$ used in the statement of Theorem~\ref{thm:main}
denotes an alternating string $xyxy \cdots$ of length $n$ beginning with $x$, 
as explained in Section~\ref{sec:notation}.
Throughout this article, $G$ will be as defined in the statement of
Theorem~\ref{thm:main}.

The faat any rank 3 Artin group has quadratic solution to the word problem
follows immediately from the combination of
Theorem~\ref{thm:main} and the proofs of  
quadratic bound for the word problem in spherical type Artin groups 
(derived in \cite{DehornoyParis} for all Garside groups), 
in large type groups (derived in \cite{HR12} as a consequence of their
automaticity), and for 3-free Artin groups (derived in \cite{BCMW}).
However we note that this result is already a consequence of the combination of Corollary G and Theorem D of a very recent preprint of Haettel and Huang \cite{HaettelHuang}.

Example~\ref{eg:n_atleast5} below shows that the condition that  $n$ be at
least $5$ is necessary for Theorem~\ref{thm:main} to hold.

That the word problem in the group $G$ is soluble was already known; it is 
rather the type of solution given by Theorem~\ref{thm:main} that we consider 
to be of interest, and what this tells us about not just the complexity of the
word problem but also the structure of geodesics in $G$.

We note that for $n>5$ the group $G$ is
\emph{locally non-spherical} (also known as \emph{2-dimensional})
and the solubility of its word problem is 
a consequence of the main theorem of \cite{Chermak}
(Chermak's algorithm reduces a word representing the identity to the empty word
using a combination of steps none of which increases its length, but does not
claim to reduce any input word to a geodesic representative); 
for $n \leq 5$ the group is of spherical type and the word problem is solved
in \cite{Deligne} (the result is generalised to Garside groups in \cite{DehornoyParis}).
However, our result shows that the method developed in \cite{HR12,HR13} and 
then extended in \cite{BCMW} can be extended further to apply to our group $G$.
The method allows any non-geodesic word to be rewritten to a geodesic in a series of 
length reducing steps, each step a sequence of length preserving moves followed by free reduction, taking linear time.

We observe that the structure of geodesics  in rank three Artin groups is an
important component in the proof of the main result of \cite{BCMWR}; the groups
$G$ for $n \geq 6$ were the only rank 3 groups that could not be included in
that result at the time of its publication, but now we have the machinery to
deal with them. 
We note also that Arye Juhasz has claimed a proof using small cancellation
theory that certain Artin groups have isoperimetric function that is at most polynomial of degree 6; his work considers all 3-free groups, as well as our
groups $G$ when $n\geq 6$ (\cite{Juhasz}, and private communication).   

The length preserving rules of the system $\cR$ that are referred to above and in 
the statement of Theorem~\ref{thm:main} generalise the $\tau$-moves
of~\cite{HR12,HR13,BCMW}. 
Each rule has the effect of replacing a \emph{critical subword} $u$ of the
input word by another word $\tau(u)$ of the same length, whose first and last
letters both differ from the corresponding letters within $u$. When a sequence
of such moves leads to a reduction of the input word, it does so by moving from
left to right within the word, replacing successive overlapping subwords, until
the letter at the righthand end of the final replacement subword is the inverse
of the subsequent letter in the word, and so a free cancellation is possible. 

In the large and sufficiently large groups studied in \cite{HR12,HR13},
the word problem can be solved as just described using critical subwords which
by definition are 2-generated with particular constraints. In \cite{BCMW} the
same type of solution was found to work with a more general definition of a
critical word that allowed it to be \emph{pseudo 2-generated} but not
necessarily 2-generated. In this article we define \emph{3-generated critical
words}, and prove that, if these are used alongside the others, then the same
type of solution works in our group $G$.

We observe that our arguments can be applied without modification to prove the
result of Theorem~\ref{thm:main} for those rank 3 Artin groups for which the
parameter $3$ is replaced by any integer greater than 3, and in that case no
$\tau$-moves to 3-generator critical words of type $\{a,b,c\}$ are necessary.
Of course in that case $G$ is 3-free, and so the results of \cite{BCMW} apply.
So this observation indicates the very close connection 
between our methods and those of \cite{BCMW}

Our article is structured as follows. After this introduction and a section
explaining our notation, we define critical words of three different types and
the $\tau$-moves that deal with them in Section~\ref{sec:critical_tau},
recalling the definitions of critical words that are 2-generated or pseudo
2-generated from \cite{HR12} and \cite{BCMW}, and then providing a definition
for 3-generated critical words. We observe some basic properties of these words
and their associated $\tau$-moves. 
Then in Section~\ref{sec:RRS} we define \emph{rightward reducing sequences} of
$\tau$-moves. We define a set $W$ to consist of all words that do not admit
such a sequence, and describe a procedure, Procedure~\ref{proc:unique_optRRS},
that can be used to reduces any word to a representative of the same element
within $W$ in quadratic time; its correctness is proved in
Proposition~\ref{prop:unique_optRRS}.
Finally Section~\ref{sec:proofs} contains the proof of Theorem~\ref{thm:main}.
That section starts with the statement of Theorem~\ref{thm:main_details}, which 
restates the result of Theorem~\ref{thm:main},
providing more detail; in particular the set $W$ is proved to be the set of all
geodesic representatives of the elements of $G$, and the set $\cR$ of
Theorem~\ref{thm:main} is found to be the set of all rightward reducing
sequences.  Technical details needed in the proof of
Theorem~\ref{thm:main_details} are provided in
Propositions~\ref{lem:6.1}--\ref{lem:6.4}, stated and proved after the proof 
of the theorem. 

We end this section with an example, which indicates why we needed to generalise
the work of \cite{BCMW} in order to deal with the group $G$.

\begin{example}
\label{eg:tricky_w_in_G}
Consider the word
$w:=(b,c)_{n-1}\cdot aba \cdot {}_{n-2}(c,b)x^{-1}$ in $G$,
where $x$ is equal to $c$ if $n$ is even or to $b$ if $n$ is odd.
A short examination reveals that $w$ admits no
rightward reducing sequence of $\tau$-moves on 2-generated or
pseudo 2-generated subwords (as defined in Section~\ref{sec:critical_tau}).
On the other hand, since $aba=_G bab$,
the word $w$ represents the same element in $G$ as
$w':=(b,c)_{n-1}\cdot bab\cdot {}_{n-2}(c,b)x^{-1}=
 (c,b)_n\cdot a\cdot {}_{n-1}(b,c)x^{-1}$,
and we can rewrite $w'$ through the rightward moving sequence of 2-generator
$\tau$-moves 
\[ (c,b)_n \rightarrow (b,c)_n,\,ca\rightarrow ac,\,{}_n(c,b)
  \rightarrow {}_n(b,c), \]
through which we have
\begin{eqnarray*}
	w' &\rightarrow& (b,c)_n\cdot a \cdot {}_{n-1}(b,c)x^{-1} \\
	&\rightarrow& (c,b)_{n-1} \cdot ac \cdot {}_{n-1}(b,c)x^{-1}
	= (c,b)_{n-1}\cdot a \cdot {}_n(c,b)x^{-1}\\
	&\rightarrow& (c,b)_{n-1}\cdot a \cdot {}_n(b,c)x^{-1},
\end{eqnarray*}
to derive from $w'$ a word ending with  the subword $xx^{-1}$.  There is
certainly a sequence of 2-generator $\tau$-moves followed by a free
cancellation that reduces $w$ (via $w'$)
to a shorter word representing the same element;
but it is not a sequence of moves that moves 
rightward within the word. There exists no such sequence, even using the 
more general $\tau$-moves of \cite{BCMW}.
We note that the proofs in \cite{HR12,HR13,BCMW} that the rewrite
systems of those articles reduce a word  to the empty word $\emptyword$ 
if and only if it represents the identity rely on the fact that $\tau$-moves
are applied in rightward reducing sequences.

In fact, the prefix $(b,c)_{n-1}\cdot aba$ of $w$ is a 3-generator 
critical subword in the sense of this article, and the move that 
replaces this by $(c,b)_{n-1}\cdot acb$ is the first step of a rightward moving
sequence of the more general $\tau$-moves of this article that leads to
reduction of $w$, as we shall see in Example~\ref{eg:tricky_w_in_G_2}
\end{example}

This article was motivated by the work of \cite{BCMW}, 
which provided a first generalisation of the methods of \cite{HR12,HR13}.  Our
proof is inspired by the proof of that article, and although this article can be read without
reference to the arguments of \cite{BCMW}, we have deliberately chosen notation
that is as close to the notation of \cite{BCMW} as was possible,  and will
refer to parts of that article at times, in order to aid comparison.

{\bf Acknowledgement:} Both authors acknowledge support from the
Leverhulme foundation grant no.RPG-2022-025 during the work on this article.

\section{Notation}\label{sec:notation}

We define $S := \{a,b,c\}$, the set of standard generators of the Artin group
$G$, and $A:= S \cup S^{-1}$, the set of standard generators and their inverses.

For $X \subseteq A$, we defined a \emph{word over $X$} to be a string of elements from $X \cup X^{-1}$; the element of that string are called the 
\emph{letters} of $w$. The \emph{length} of a word $w$ is its length as a
string. Elements of $S=\{a,b,c\}$ are called \emph{positive} generators,
and elements of $S^{-1}:= \{a^{-1},b^{-1},c^{-1}\}$ negative generators;
a word is called positive if it involves only positive generators,
negative if it involves only negative generators, and otherwise is called
\emph{unsigned}; a letter could be positive or negative.
We define the \emph{name} of a letter $x$ to be the positive generator within $\{x,x^{-1}\}$ to which it corresponds.
A word is called \emph{geodesic in $G$} if it has minimal length as a representative
of the element of $G$ that it represents.

For $x,y,z \in A$ with distinct names, we refer to 
words over $\{x,y\}$
as \emph{ $\{x,y\}$-words}, also as \emph{$2$-generator words},
and we refer to words over $\{x,y,z\}$
as \emph{ $\{x,y,z\}$-words}, also as \emph{3-generator words}.

For $x,y \in A$, as above, we define
${}_{n}(x,y)$ to be
the alternating string $xyxy \cdots$ of length $n$ beginning with $x$ and
$(y,x)_n$
to be the alternating string $\cdots yxyx$ of length $n$ ending with $x$
If $x,y$ are both positive, such a string is called \emph{positive alternating}
of length $n$, while
if $x,y$ are both negative, it is called \emph{negative alternating} of length $n$.

For any word $w$, we denote by $\f{w}$ and $\l{w}$ its first and last letter,
and by $\p{w}$ and $\s{w}$ its maximal proper prefix and maximal proper suffix.
Note that $w=\f{w}\s{w}=\p{w}\l{w}$.

\section{Critical words and $\tau$-moves}

\label{sec:critical_tau}
The concept of a critical word as a particular type of 2-generator word in an
Artin group was introduced in \cite{HR12},
and in that article and \cite{HR13} those words were an essential part of a
process for reduction in large and sufficiently large Artin groups. 
In this article we shall need also to define critical 3-generator words; these
will be of two kinds: pseudo 2-generated words (introduced in \cite{BCMW}) and
3-generator critical words of type $\{a,b,c\}$,
both defined in Section~\ref{sec:3gen} below.

\subsection{Critical words and $\tau$-moves involving just 2 generators}
\label{sec:2gen}

Critical 2-generator words are geodesic in 2-generated Artin groups 
\cite{MairesseMatheus} and vital to the recognition of and reduction to geodesics in those Artin groups that are studied in \cite{HR12,HR13,BCMW}.
We recall the definition and some technical results from \cite{HR12,HR13} below,
as well as some further technical lemmas that we shall need in this article;
the proofs could easily be omitted from a first reading.

Let $X=\{x,y\}$ be a subset of $S$ with $x \neq y$, and let $m_{xy}$ be the
integer associated with $x,y$ in the presentation of $G$.
The element $\Delta$
of the subgroup $\langle x,y \rangle$ of $G$ that is represented by the 
alternating product $(x,y)_{m_{xy}}$ 
is known as the \emph{Garside element} of $\langle x,y \rangle$ and is 
also represented by the other alternating product $(y,x)_{m_{xy}}$.
Let $v$ be a word $w$ over $X$.
Following \cite{HR12} we  define $p(v)$ to be the minimum of $m_{xy}$ and the 
maximal length of a positive alternating subword of $v$,
and
$n(v)$ to be the minimum of $m_{xy}$ and the 
maximal length of a negative alternating subword of $v$.

\begin{definition}[2-generator critical words]
	\label{gen:2gen_critical}
The 2-generator word $u$ is defined to be \emph{critical} if
\begin{mylist}
\item[(i)] $p(u)+n(u)=m_{xy}$; and
\item[(ii)] either 
	\begin{mylist}
	\item[(a)] $u$ is positive of the form ${}_{m_{xy}}(x,y)\xi$
	or $\xi(x,y)_{m_{xy}}$
	where $p(\xi)<{m_{xy}}$, 
or \item[(b)]
	$u$ is negative of the form ${}_{m_{xy}}(x^{-1},y^{-1})\xi$ 
	or $\xi(x^{-1},y^{-1})_{m_{xy}}$ 
	where $n(\xi)<{m_{xy}}$, 
or \item[(c)] $u$ is unsigned, and has one of the two forms
	${}_{p(u)}(x,y)\xi (z^{-1},t^{-1})_{n(u)}$
	or 
	${}_{n(u)}(x^{-1},y^{-1})(x,y)\xi (z,t)_{p(u)}$,
	where $\{z,t\}=\{x,y\}$.
	\end{mylist}
\end{mylist}
\end{definition}

\begin{lemma}\label{lem:critlintest}
We can check in linear time whether a $2$-generator word is critical.
\end{lemma}
\begin{proof} For any $w$, We can calculate $p(w)$ and $n(w)$. We can also test whether a
signed word is critical in a single scan of the word. 
\end{proof}

\begin{lemma}\label{lem:critsubword}
Suppose that the critical $2$-generator word $u$ has a critical subword $v$,
and so $u = u_pvu_s$ for some words $u_p$, $u_s$. Then the word $vu_s$ has a
critical suffix.
\end{lemma}
\begin{proof}
Since $p(u)+n(u) = p(v)+n(v) = m_{xy}$ with $p(v) \le p(u)$ and $n(v) \le n(u)$,
we must have $p(u)=p(v)=p(vu_s)$ and $n(u)=n(v)= n(vu_s)$, and the result
follows easily. 
\end{proof}
 
Still following \cite{HR12}, we define an involution $\tau$ on
the set of all 2-generator critical words as follows.

\begin{definition}[2-generator $\tau$-moves]
First we define $\delta(x)$, to be the element of $X \cup X^{-1}$ that is equal in $G$ to the conjugate $\Delta x \Delta$ of $x$, and similarly we define $\delta(y)$, $\delta(x^{-1})$, $\delta(y^{-1})$. We see that if $m_{xy}$ is even then $\delta$ is the identity map, but otherwise $\delta$ has order 2 as a permutation of $X \cup X^{-1}$, swapping $x$ and $y$, and $x^{-1}$ and $y^{-1}$. Then for any word $\xi$ over $X$, we define
$\delta(\xi)$ to be the product of the letters that are the images under $\delta$ of the letters of $\xi$.

Now, for unsigned 2-generator critical words, we define $\tau$ by
\begin{eqnarray*}
\tau({}_p(x,y)\,\xi\,(z^{-1},t^{-1})_n)
&:=& {}_n(y^{-1},x^{-1})\,\delta(\xi )\,(t,z)_p,\\
\tau({}_n(x^{-1},y^{-1})\,\xi\,(z,t)_p)
&:=& {}_p(y,x)\,\delta(\xi )\,(t^{-1},z^{-1})_n.
\end{eqnarray*}

Then, for positive and negative 2-generator critical words, we define $\tau$ as
follows, where $\xi$ is assumed non-empty in the final four equations.
\begin{eqnarray*}
	\tau({}_{m_{xy}}(x,y))&:=& {}_{m_{xy}}(y,x),\\
	\tau({}_{m_{xy}}(x^{-1},y^{-1}))&:=& {}_{m_{xy}}(y^{-1},x^{-1})\\
	\tau({}_{m_{xy}}(x,y)\,\xi) &:=& \delta(\xi)\,(z,t)_{m_{xy}},
\quad\hbox{\rm where}\quad z=\l{\xi},\,\{x,y\}=\{z,t\},\\
	\tau(\xi\,(x,y)_{m_{xy}}) &:=& {}_{m_{xy}}(t,z)\,\delta(\xi),
\quad\hbox{\rm where}\quad z=\f{\xi},\,\{x,y\}=\{z,t\},\\
	\tau({}_{m_{xy}}(x^{-1},y^{-1})\,\xi) &:=& \delta(\xi)\,(z^{-1},t^{-1})_{m_{xy}},
\quad\hbox{\rm where}\quad z=\l{\xi}^{-1},\,\{x,y\}=\{z,t\},\\
	\tau(\xi\,(x^{-1},y^{-1})_{m_{xy}}) &:=& {}_{m_{xy}}(t^{-1},z^{-1})\,\delta(\xi),
\quad\hbox{\rm where}\quad z=\f{\xi}^{-1},\,\{x,y\}=\{z,t\}.
\end{eqnarray*}
\end{definition}

It is verified in \cite{HR12} that a 2-generator critical word $u$ and its
image $\tau(u)$ represent the same element of the subgroup
$\langle x,y\rangle$ of $G$; hence $u=_G \tau(u)$. Replacement of a critical
subword $u$ of a (probably) longer word by $\tau(u)$ is called a $\tau$-move.

In  this article we shall often refer to the $\tau$-moves we have just defined
on 2-generator critical words as \emph{2-generator $\tau$-moves} in order to
distinguish them from more general $\tau$-moves defined in \cite{BCMW} for
P2G words and in this article for (other) 3-generator words.

\begin{lemma}\label{lem:2gengeo}
Let $w$ and $w'$ be geodesic words over the standard generators of a
$2$-generator Artin group that represent the same element of the group. Then we
can transform $w$ to $w'$ using a sequence of 2-generator $\tau$-moves.
\end{lemma}
\begin{proof}  By \cite[Theorem 2.4]{HR12} we can transform each of $w$ and $w'$
to their (equal) shortlex normal forms using a sequence of (2-generator)
$\tau$-moves. Since the reverse of a 2-generator $\tau$-move is also a 
2-generator $\tau$-move, the result follows.
\end{proof}

We observe the following facts about 2-generator critical words and
$\tau$-moves, which we shall need later. 
\begin{lemma}
	\label{lem:2gen_taufacts}
For any 2-generator critical word $u$, the names of the first letters of $u$
and $\tau(u)$ are distinct, as are the names of the last letters of
$u$ and $\tau(u)$.
\end{lemma}
\begin{proof}
This proved as part of \cite[Proposition 2.1]{HR12}
\end{proof}

\begin{lemma}\label{lem:abcrit1}
Let $v$ be an $\{a,b\}$-word such that $\f{v}$ and $\l{v}$ both have name $a$,
and suppose that $v$ can be transformed using a sequence of $\tau$-moves to a
word $b^{i}a^{j}b^{k}$ with $i,j,k$ nonzero integers.
Then $v$ is a critical word.

Furthermore, if $v$ is a positive or a negative word then $|j|=1$, and either
(i) $|i|=1$ and $v = a^k b^j a^j$,
or (ii) $|k|=1$ and $v = a^j b^j a^i$.
\end{lemma}
\begin{proof} By \cite[Proposition 4.3]{MairesseMatheus}, a word $w$ in the
$2$-generator Artin group $\langle a,b \mid aba=bab \rangle$ is geodesic if and
only if $p(w) + n(w) \le 3$, and there is a second geodesic word
defining the same group element if and only if $p(w)+n(w)=3$.
Since the word $w:= b^{i}a^{j}b^{k}$ satisfies $p(w)+n(w) \le 3$ for all values of
$i,j,k$, this word is geodesic and hence so is $v$, and since $v \ne w$ we must
have $p(v)+n(v) = 3$ with $p(v) = p(b^{i}a^{j}b^{k})$ and
$n(v) = n(b^{i}a^{j}b^{k})$.

Suppose first that $v$ is a positive word, so $p(v)=3$ and $n(v)=0$. The
case when $v$ is negative is similar. So $\f{v}=\l{v}=a$.
Since $p(b^{i}a^{j}b^{k}) = 3$, we must have $j=1$ and $i,k \ge 1$.
We claim that the only positive words that are equal
in $G_{ab}$ to $b^{i}a^{j}b^{k}=b^i a b^k$ are the words of the form
$b^{i'}aba^{i-i'}b^{k-1}$
for $0 \le i' \le i$ and $b^{i-1} a^{k-k'}bab^{k'}$ for $0 \le k' \le k$.
This claim can be verified by checking that any substitution of $aba$ by $bab$
or vice versa results in the replacement of one word of this form by another.
It follows from the claim, together with $\f{v} = \l{v}=a$,
that $k=1$ or $i=1$
with $v = aba^i$ or $v = a^k ba$, both of which are critical words,

Otherwise $v$ is unsigned, so either $p(v)=2$ and $n(v)=1$, or $p(v)=1$ and
$n(v)=2$. Suppose that $p(v)=2$ and $n(v)=1$; the other case is similar.
Since the same applies to the word $b^{i}a^{j}b^{k}$ we must have either
$i,j>0$ and $k<0$, or $i<0$ and $j,k>0$.
Assume the former---again the other
case is similar. Then the word $b^{-1}v$ is not geodesic in $G_{ab}$ so,
using the results of~\cite{MairesseMatheus} again, we must have $p(b^{-1}v) =
n(b^{-1}v) = 2$, so $\f{v}=a^{-1}$. Similarly, since the word $vb$ is
non-geodesic, we must have $p(vb)=3$ and $n(vb)=1$, so $v$ has
the suffix $ba$, and hence $v$ is critical as claimed.
\end{proof}

\begin{lemma}\label{lem:abcrit2}
Let $v$ be an $\{a,b\}$-word such that $\f{v}$ and $\l{v}$ both have name $a$.
Then we can decide in linear time whether $v$ can be transformed using a
sequence of $\tau$-moves to a word $b^ia^jb^k$ with $i,j,k$ nonzero integers
and, if so, compute the resulting transformed word.
\end{lemma}
\begin{proof}
We can calculate $p(v)$ and $n(v)$ in linear time and if $p(v) + n(v) \ne 3$
then no such transformation is possible, and we return \false.
So we can assume that $p(v)+n(v)=3$ and hence that $v$ is a geodesic word.
If $v$ is a positive word then, by Lemma~\ref{lem:abcrit1}, the
transformation in question exists if and only if either $j=k=1$ and $v=a^kba$,
or $i=j=1$ and $v = aba^i$, which we can test directly in linear time.
The case when $v$ is a negative word is similar.

So suppose that $v$ is unsigned. Then $v$ is critical by
Lemma~\ref{lem:abcrit1}. We describe a linear-time process that will transform
$v$ to a word of the form $b^ia^jb^k$ and return \true\ if possible, and
return $\false$ otherwise. We shall not discuss how this process could be
implemented in linear time, except to note that such an implementation might
involve storing the word $v$ and the transformed word as doubly linked lists.

We read the word $v$ from both ends simultaneously, i.e. leftwards from
the beginning of $v$ and rightwards from its end, and we shall denote the
central part of $v$ that has not yet been read by $v'$. We shall maintain
a variable \flip\ that is \true\ or \false, and initially \false. When \flip\ 
is true the letters $a$ and $a^{-1}$ of $v'$ are regarded as being
replaced $b$ and $b^{-1}$ respectively, and vice versa. We shall call the
letters of the word $v'$ after making this replacement if necessary the
\emph{interpreted} letters of $v'$ (which are the same as the letters of $v'$
itself when \flip\ is \false).

At the beginning of the process we are reading the first and last letters
of $v$ which, since $v$ is critical, must be $a$ and $a^{-1}$, not
necessarily respectively. Suppose, for example, that $v$ has the prefix
$ab$ and the suffix $a^{-1}$; all other cases are similar. Then we replace
the prefix $ab$ by $b^{-1}$ and the suffix $a^{-1}$ by $ab$. We then set
the flag \flip\ to \true. Note that, if we were to replace the letters of
the unread part $v'$ of $v$ by their interpreted letters at this stage,
then we would have executed the $\tau$-move on $v$, but we do not do that
explicitly.

As we continue to read $v'$ from both ends, if an interpreted letter has name
$b$, then we replace the original letter (if necessary) by the interpreted
letter, and we do this until the interpreted letters at both ends of $v'$ have
name $a$. If this never happens (i.e. if $v'$ becomes empty without this
occurring) then we return \false. Otherwise, if the two interpreted letters
are equal, then we check whether the remainder of the interpreted word
$v'$ is just a power of $a$, in which case, after replacing $v'$ by its
interpreted word if necessary, the resulting transformed word has the
required from, so we return \true. If it is not a power of $a$ then we return
\false.

On the other hand, if the interpreted letters at the beginning and end of $v'$
are $a$ and $a^{-1}$ then, by Lemma~\ref{lem:abcrit1}, if the required
transformation is possible, then the interpreted word $v'$ must be critical,
which we can test for by looking at the second letters from the beginning
and end of $v'$.  If it is not critical then we return \false. Otherwise we
restart the process at this point, by making the necessary changes to the
beginning and end of $v'$, and then changing the value of \flip.
\end{proof}

\subsection{Critical words and $\tau$-moves involving more than 2 generators}
\label{sec:3gen}
In addition to the 2-generator critical words defined in Section ~\ref{sec:2gen},
we shall need two kinds of critical words involving 3 generators, both words that are \emph{pseudo 2-generated} as defined in \cite{BCMW} and some further
words that we shall call $\{a,b,c\}$-critical, and define below.
In the article \cite{BCMW} a word that is \emph{pseudo 2-generated} can 
involve any number of generators in addition to its selected pair of 
non-commuting \emph{pseudo generators}; but 
in this article a pseudo 2-generated (P2G) word $u$ must have  
pseudo-generators $\{x,b\}$, where either $x=a$ or $x=c$ and at most one further generator $z$, commuting with $x$, and so equal to $c$ or $a$.
We call such words \emph{P2G words of type $\{x,b\}$}
and define them now.
Since there is only one further generator within $G$, we can simplify some
details of the definition of P2G words from \cite{BCMW}.

\begin{definition}[P2G words of types $\{a,b\}$ and $\{b,c\}$]
	\label{def:P2G}
	Where $x=a$ or $c$, and $z=c$ or $a$ (respectively),
a P2G word of type $\{x,b\}$ in $G$ is defined to have the form
$u = u_pu_qu_s$, where
\begin{mylist}
\item[(i)] 
	either $\f{u}=\f{u_p} \in \{x,x^{-1}\}$ and
	$u_p$  is an $\{a,c\}$-word,
	or $\f{u}=\f{u_p} \in \{b,b^{-1}\}$ and
	$u_p$ is a power of $b$;
\item[(ii)] $u_q$ is an $\{x,b \}$-word, 
starting with a letter whose name is not that of $\f{u}$, 
	and ending with a letter whose name is not that of $\l{u}$;
\item[(iii)] 
	either $\l{u}=\l{u_s} \in \{x,x^{-1}\}$ and 
	$u_s$ is an $\{a,c\}$-word or
$\l{u} = \l{u_s} \in \{b,b^{-1} \}$, and
	$u_s$ is a power of $b$ 
\end{mylist}
For such a word we define $\alpha(u):= z^j$, where $j$ is the total power
of $z$ in $u_p$ and $\beta(u) := z^k$, where $k$ is the total power of $z$ 
in $u_r$. 
\end{definition}

For a P2G word $u$, we shall denote the word obtained by deleting all
letters from $u$ other than its two pseudo-generators and their inverses by
$\widehat{u}$. (If $u$ involves only its two pseudo-generators, then $u=\widehat{u}$).

\begin{definition}[P2G, criticality and $\tau$-moves]
	\label{def:P2G_critical_tau}
A P2G word $u$ is said to be \emph{critical} if $\widehat{u}$ is a critical
$2$-generator word over the pseudo-generators 
	$\{x,b\}$.
In that case, still following \cite{BCMW}, we extend the definitions of
$\tau$-moves defined in \cite{HR12} and \cite{HR13} to apply to $u$, and define
\[\tau(u) := \alpha(u)\tau(\widehat{u})\beta(u),\] 
where $\alpha(u),\beta(u)$ and the 2-generator critical word $\widehat{u}$ are
as defined above, and $\tau(\widehat{u})$ is as defined in
Section~\ref{sec:2gen}.

We shall refer to the replacement by $\tau(u)$ of a critical P2G word $u$ of
one of the two types that we have just defined as a \emph{$\tau$-move},
\emph{of type $(a,b)$} or \emph{of type $(b,c)$}.
\end{definition}

\begin{example}
	\label{eg:P2G_critical_tau}
The words $ab^3a^{-1}$ and $bc^2bcb^{-1}c^{-1}$ are 2-generator critical words
with $\tau(ab^3a^{-1})=b^{-1}a^3b$ and
$\tau(bc^2bcb^{-1}c^{-1})=c^{-1}b^{-1}cbc^2b$.

The word $u_1:= acb^3c^{-1}a^{-1}$ is P2G-critical of type $\{a,b\}$,
with $\widehat{u_1} = ab^3 a^{-1}$, $\alpha{u_1}=c$,
$\beta(u_1)=c^{-1}$,$\tau(u_1) = cb^{-1}a^3bc^{-1}$.

The word $u_2:= bc^2bcb^{-1}a^2c^{-1}$ is P2G-critical of type $\{b,c\}$,
with $\widehat{u_2} = bc^2bcb^{-1}c^{-1}$, $\alpha{u_2}=\emptyword$,
$\beta(u_2)=a^2$, $\tau(u_2) = c^{-1}b^{-1}cbc^2ba^2$.
\end{example}

We note that any $\tau$-move for a critical P2G word decomposes as a 
sequence of $\tau$-moves and commutator moves for 2-generator subwords, but in general that sequence of 2-generator moves does not move rightward within the word.

On the subset of P2G words that are actually 2-generated, the definitions of
critical words and of $\tau$-moves coincide with those already given for
2-generator words.

Note that our definitions of $\alpha(u)$ and $\beta(u)$ follow those of
\cite{BCMW}, but in that article, where the groups may have more than 3 generators, the words $\alpha(u)$ and $\beta(u)$ may involve more than one generator.
Since in our group $G$ there are no generators that commute with two other
generators, the subword denoted by $\rho(u)$ in \cite{BCMW} is empty.

We observe the following about P2G-critical words in $G$, and  their associated
$\tau$-moves.
\begin{lemma}
	\label{lem:P2G_taufacts}
Let $u$ be a P2G-critical word $u$, with pseudo-generators $\{x,b\}$,
where $x=a$ or $c$.
Then
\begin{mylist}
\item[(i)] the first and last letters of $u$ are in $\{x,x^{-1},b,b^{-1}\}$,
	and equal to the first and last letters of $\widehat{u}$,
\item[(ii)] the names of the first letters of $u$ and $\tau(u)$ are distinct,
as are the names of the last letters of $u$ and $\tau(u)$,
\item[(iii)] $\f{\tau(u)} \in \{x,b\} \iff \alpha(u) = \emptyword$,
\item[(iv)] $\l{\tau(u)} \in \{x,b\} \iff \beta(u) = \emptyword$,
\end{mylist}
\end{lemma}
\begin{proof}
This can be observed from the definition together with the properties observed
in Lemma~\ref{lem:2gen_taufacts}.
\end{proof}

Note that the powers $\alpha(u)$ and $\beta(u)$ of the third generator
introduce a lack of symmetry into the definition of a critical P2G word $u$, 
and if either of those two words is non-empty then $\tau(u)$ is not itself
critical.

The example $w=(b,c)_{n-1}\cdot aba \cdot {}_{n-2}(c,b)x^{-1}$
considered in Example~\ref{eg:tricky_w_in_G} for our selected group $G$ still does not admit a
\emph{rightward reducing sequence} (RRS) as defined in \cite{BCMW}.
To remedy this,
we shall extend our definition of a critical word that is 2-generated over
$\{b,c\}$ to allow not only P2G words but also some further 3-generator words
over $S$ to be considered critical, each one constructed out of two P2G words,
one of type $\{a,b\}$ and the other of type $\{b,c\}$, as explained below. Then we define
$\tau$-moves on these new critical words, with the
intention that, by using these new $\tau$-moves together with the ones
from \cite{BCMW} defined above, we can find rightward 
reducing sequences to reduce any non-geodesic word over $S$.

\begin{definition}[3-generator critical word]
	\label{def:3gen_critical}
We call a 3-generator word over $S$ critical if it has the form
$u := u_p u_q u_r$,
where
\begin{mylist}
\item[(i)]
either $\f{u}=\f{u_p} \in \{b,b^{-1}\}$ and
$u_p$ is a power of $b$; or
or $\f{u}=\f{u_p} \in \{c,c^{-1}\}$ and
$u_p$  is an $\{a,c\}$-word;
\item[(ii)] $u_q$ is a $\{b,c\}$-word,
starting with a letter whose name is different from that of $\f{u}$;
\item[(iii)] $u_r$ is a P2G word of type $\{a,b\}$ with
$\f{u_r},\l{u_r} \in \{a,a^{-1}\}$, for which the $\{a,b\}$-word
$\widehat{u_r}$ can be transformed, using a sequence of $\tau$-moves, into a
word of the form $b^{\ii} a^{\jj} b^{\kk}$ for some nonzero
$\ii,\jj,\kk \in \Z$.
\item[(iv)] $u^\# := u_p u_q \alpha(u_r) b^{\ii}$ is a P2G-critical word of
type $\{b,c\}$.
\end{mylist}
We shall call such words \emph{critical of type $\{a,b,c\}$}.
Note that the word $u_r$ in (iii) must be critical by Lemma~\ref{lem:abcrit1}.
\end{definition}

Note that $u^\#$ is a prefix of a word that can be derived from $u$ by a 
combination of commutation rules between $a$ and $c$ and $\tau$-moves on the 
2-generated word $\widehat{u_r}$.

Note also that, by Lemma~\ref{lem:2gengeo} above, the condition that
$\widehat{u_r}$ can be converted to
$b^{\ii}a^{\jj}b^{\kk}$ using a succession of $\tau$-moves on $\{ a,b \}$-words is equivalent to it being geodesic in the 2-generated Artin group.

Applying first commutation rules between $a$ and $c$, then the relation
$\widehat{u_r} =_G b^{\ii}a^{\jj}b^{\kk}$, and then more commutation rules
between $a$ and $c$, we see that
\begin{eqnarray*}
u &=& u_p u_q u_r =_G u_pu_q \alpha(u_r) \widehat{u_r} \beta(u_r)\\
&=_G& u_p u_q \alpha(u_r) b^{\ii} a^{\jj} b^{\kk} \beta(u_r) 
= u^\# a^{\jj} b^{\kk} \beta(u_r) \\
&=_G& \alpha(u^\#) \widehat{u^\#} a^{\jj} b^{\kk} \beta(u_r) \end{eqnarray*}

We have $\tau(\widehat{u^\#})$  equal to $\p{\tau(\widehat{u^\#})}c^\ee$,
where $\ee = \pm 1$.
So now,  applying $\tau$ to $\widehat{u^\#}$,
followed by more commutation rules, we derive 
\begin{eqnarray*}
u &=_G& \alpha(u^\#) \tau(\widehat{u^\#}) a^{\jj} b^{\kk} \beta(u_r) 
= \alpha(u^\#)  \p{\tau(\widehat{u^\#})} c^{\ee} a^{\jj} b^{\kk} \beta(u_r)\\
&=_G& \alpha(u^\#)  \p{\tau(\widehat{u^\#})} a^{\jj} c^{\ee} b^{\kk} \beta(u_r).
\end{eqnarray*}

For a 3-generator critical word $u$, we define $\alpha(u)$ and $\beta(u)$ to be
the powers $\alpha(u^\#)$ and $\beta(u_r)$ of $a$ and $c$ respectively 
associated with the P2G words $u^\#$ and $u_r$ that are identified above.  

\begin{definition}[$\tau$-moves of type $\{a,b,c\}$]
	\label{def:3gen_tau}
In the notation defined above, for the 3-generator critical word
$u = u_p u_q w_r$ of type $\{a, b, c\}$, we define
\[ \tau(u) := \alpha(u)  \p{\tau(\widehat{u^\#})} a^{\jj} c^{\ee} b^{\kk}
\beta(u).\] 
We call these new moves \emph{$\tau$-moves of type $(a,b,c)$}.
\end{definition}
Note that in order to verify this identity we have applied $\tau$-moves on
2-generated subwords moving from right to left within the word,
rather than from left to right.  
This is why
we need to collect this set of moves together
into a single $\tau$-move on the 3-generator critical word.

\begin{example}
	\label{eg:3gen_critical_tau}
The word $u:= (b,c)_{n-1}\cdot aba$ of $w$ from Example~\ref{eg:tricky_w_in_G}
is 3-generator critical,
with ${u}_p$ a single generator, 
\begin{eqnarray*}
	&&{u}_q = (b,c)_{n-2},\,{u}_r = aba=\widehat{{u}_r}=_G bab.\\
\hbox{\rm Then}&& u^\#=(b,c)_{n-1}b,\end{eqnarray*}
$\alpha(u)$ and $\beta(u)$ are both empty and the integers $\ii,\jj,\kk$ of
Definition~\ref{def:3gen_critical}(iii) are all equal to 1.

Further $\tau(u^\#)=(b,c)_n=(c,b)_{n-1}c$, So $\ee=1$.
It follows that $\tau(u) = (c,b)_{n-1}acb$.
\end{example}

	
As in \cite{HR12,HR13,BCMW},
we shall reduce words in our group $G$ using particular sequences of 
overlapping $\tau$-moves which (again as in \cite{HR12,HR13,BCMW}) we shall
call \emph{rightward reducing sequences}.

\section{Rightward reducing sequences of $\tau$-moves}
\label{sec:RRS}

In order to reduce a word $w$ using a sequence of $\tau$-moves (including the
new moves defined in Section~\ref{sec:critical_tau}),
we shall write $w$ as a concatenation of subwords, to which we associate
a sequence of words in $A^*$ that we shall call a
\emph{rightward reducing sequence}, abbreviated as RRS.
Our definition of an RRS extends the definition of \cite{BCMW},
which was already based on a definition within \cite{HR12,HR13}.
We use notation similar to that of \cite{BCMW} to facilitate comparison.

An RRS will be a sequence $U=u_1,\ldots,u_m,u_{m+1}$ of words in $A^*$,
with $u_i$ critical for $i \leq m$, 
that is associated with a factorisation of 
a word $w$ that is to be reduced, as we now describe.

\begin{definition}[Rightward reducing sequence (RRS)]
	\label{def:RRS}
Let
$U=u_1,\ldots,u_m,u_{m+1}$ be a sequence of words over $A$,
for which each $u_i$ with $1 \le i \le m$ 
is either a P2G-critical word of type $\{a,b\}$ or $\{b,c\}$ 
or a critical word of type $\{a,b,c\}$,
and $u_{m+1} = xv$, where the letter $x$ commutes with every letter
in the word $v$. 
For $1 \le i \le m$,  write $\alpha_i$ and
$\beta_i$ for $\alpha(u_i)$ and $\beta(u_i)$ respectively.

For a word $w$ over $A$, we say that \emph{$w$ admits $U$ as an RRS of
length $m$} if $w$ can be written as a concatenation of words
$w=\mu w_1 \cdots w_mw_{m+1}\gamma$, with $w_i$ 
		non-empty for $1 \le i \le m$, $u_1=w_1$, 
$\f{u_{m+1}} = \f{\gamma}^{-1}$ and, for each $i$ with $2 \le i \le m+1$,
	$u_i$ is determined (in terms of $u_{i-1}$, $w_i$) according to case (i) or (ii) below.
\begin{mylist}
\item[(i)] if $u_{i-1}$ is a P2G-critical word,  
	then $u_i = \l{\tau(\widehat{u_{i-1}})} \beta_{i-1} w_i$;
\item[(ii)] if $u_{i-1}$ is a critical word of type $\{a,b,c\}$, then
 $u_i = c^{\ee} b^{\kk}  \beta_{i-1} w_i$,
 using the notation of Definitions~\ref{def:3gen_critical}
 and~\ref{def:3gen_tau}.
\end{mylist}
Note that in (i), the notation $\widehat{u_{i-1}}$ is used, just as
$\widehat{u}$ was used earlier, to denote the 2-generator word over the
2-pseudo generators formed by deleting all occurrences of the third generator. 

We call the factorisation $\mu w_1 \cdots w_mw_{m+1}\gamma$ of $w$
the \emph{decomposition associated with $U$}.
\end{definition}

Note that the word $w_{m+1}$ can be empty if $m>0$. If $m=0$ then
$w_{m+1}=w_1=\f{\gamma}^{-1}v$.  So a word that is not freely reduced
admits an RRS of length 0, with $v=\emptyword$.

\begin{lemma}
\label{lem:RRSdetails}
Let $U=u_1,\ldots,u_m,u_{m+1}$ with $m>0$ be an RRS for a word $w$ of $G$
with associated decomposition $w=\mu w_1 \cdots w_mw_{m+1}\gamma$,
and let $\alpha_i:= \alpha(u_i),\beta_i:= \beta(u_i)$, for each $i$. Then
\begin{mylist}
\item[(i)] $u_m$ cannot have type $\{a,b,c\}$;
\item[(ii)] $\beta_m= \emptyword$;\quad
(iii) for $i<m$,
$\beta_i \ne \emptyword \Rightarrow  \alpha_{i+1} = \emptyword$.
\item[(iv)] if $u_i$ is critical of type $\{a,b,c\}$ for some $i<m$,
then $\alpha_{i+1} = \emptyword$, and
$u_{i+1}$ must be either a P2G-critical word of
type $\{b,c\}$ or a critical word of type $\{a,b,c\}$. 
\end{mylist}
\end{lemma}
\begin{proof} (partly from \cite[Lemma 3.10]{BCMW}).
(i) If $u_m$ had type $\{a,b,c\}$ then $\f{u_{m+1}} = c^{\ee}$ would not
commute with the second letter of $u_{m+1}$, which has name $b$, contrary to
assumption.

(ii) The letters in $\beta_m$, which is a subword of $u_{m+1}$,
do not commute with $\f{u_{m+1}} = \l{\tau(\widehat{u_m})}$,
so $\beta_m = \emptyword$.

(iii) If $\beta_i \ne \emptyword$ then the first two letters of $u_{i+1}$ do
not commute, so $\alpha_{i+1} = \emptyword$.

(iv) If $u_i$ has type $\{a,b,c\}$ then the first two letters have names $b$
and $c$, so $\alpha_{i+1} = \emptyword$ and $u_{i+1}$ cannot have type
$\{a,b\}$.
\end{proof}

As in \cite{BCMW}, we can think of an RRS as a way of applying successive
$\tau$-moves and commutations, moving from left to right, until a free
reduction becomes possible.
To apply the RRS $U$ to $w$, we first replace the subword $u_1=w_1$ by
$\tau(u_1)$ (or by $\s{w_1}\f{w_1}$ when $m=0$). 
For $m>0$, this results in
a word with subword $u_2$ consisting of a suffix of $\tau(u_1)$ followed
by $w_2$. When $m>1$ we replace $u_2$ by $\tau(u_2)$ and continue like
this, 
replacing $u_i$ by $\tau(u_i)$ for each $i\leq m$, 
finally replacing $u_{m+1}$ by $\s{u_{m+1}}\f{u_{m+1}}$. The
resulting word has the freely reducing subword $\f{\gamma}^{-1}\f{\gamma}$. 

\begin{lemma}\label{lem:RRScheck}
We can check in linear time whether a factorisation
$w=\mu w_1 \cdots w_mw_{m+1}x$ of a word $w$ is associated with an RRS $U$
of $w$.
\end{lemma}
\begin{proof} By Lemma~\ref{lem:abcrit2}, we can check in linear time whether
a given word could occur as the suffix $(u_i)_r$ of a critical word $u_i$ of
type $\{a,b,c\}$. Given this and Lemma~\ref{lem:critlintest},
we see easily that all types of $\tau$-moves can be
executed in linear time, and the result follows.
\end{proof}
 
\begin{example}~\label{eg:tricky_w_in_G_2}
We consider the word
$w=(b,c)_{n-1} \cdot aba \cdot  {}_{n-2}(c,b)x^{-1}$ that was considered 
in Example~\ref{eg:tricky_w_in_G}.
That admits an RRS of length 2, with associated decomposition $w_1w_2w_3x^{-1}$,
where $w_1=(b,c)_{n-1}aba$, $w_2={}_{n-2}(c,b)$ and $w_3$ is empty.

The word $w_1$ is a 3-generator critical word of type $\{a,b,c\}$,
as we observed in Example~\ref{eg:3gen_critical_tau}, and
$\tau(w_1) = (c,b)_{n-1}acb$ (as we computed in that example).

Now $u_2 = cb\cdot w_2 = {}_n(c,b)$, and $\tau(u_2)={}_n(b,c)$.
So application of the RRS to $w$ yields
\begin{eqnarray*}
w &\to& \tau((b,c)_{n-1} aba) \cdot {}_{n-2}(c,b)x^{-1}
=(c,b)_{n-1}acb\cdot  {}_{n-2}(c,b)x^{-1}\\
	&\to& (c,b)_{n-1} a \cdot  \tau({}_{n}(c,b))x^{-1} 
 = (c,b)_{n-1} a  \cdot {}_{n}(b,c)x^{-1},
\end{eqnarray*}
which ends in the cancelling pair $xx^{-1}$.
\end{example}

The following example with $n=4$ shows that we need $n$ to be at least $5$.

\begin{example} \label{eg:n_atleast5}
When $n=4$, the word $(bab)(cbca)(bc)b^{-1}$ admits the RRS
\[ bab, acbca,bcbc,bb^{-1}\] associated with the decomposition of length 3 that
corresponds to the bracketing shown, and hence allows the sequence of reductions
\begin{eqnarray*}
(bab)cbcabcb^{-1} 
	&\to& \tau(bab)cbcabcb^{-1} = ab(acbca)bcb^{-1}\\
	&\to& ab(c\tau(aba)c)bcb^{-1}=abcba(bcbc)b^{-1}\\
&\to& abcba(\tau(bcbc)=abcbacbcbb^{-1}\to abcbacbc.\end{eqnarray*}
But $bacbcbabcb^{-1}$, which clearly represents the same element of $G$, admits no RRS; its only critical subword is $bab$.
\end{example}

We define an \emph{optimal RRS} for a word in a $(2,3,n)$-group $G$ as follows 
(following \cite{BCMW}).
\begin{definition}[Optimal RRS]
\label{def:optRRS}
We say that an RRS $U=u_1,\ldots,u_{m+1}$ for $w$, with associated
decomposition $w=\mu w_1\cdots w_{m+1}\gamma$, is \emph{optimal} if
\begin{mylist}
\item[{\rm (i)}] the first letter of $w_1$ is at least as far to the right as in
any other RRS for $w$, i.e. $\mu$ is as long as we found possible,
\item[{\rm (ii)}] 
	$\f{\gamma}$ does not appear in $w_{m+1}$,
\item[{\rm (iii)}] if, for some $i>1$ with $1 < i \le m$, $u_{i-1}$ is
	a P2G word of type $\{x,y\}$ and $\alpha_i=\emptyword$, then $u_i$ is 
\begin{mylist}
\item[either] a P2G word  of type $\{z,t\}$
with $|\{x,y\} \cap \{ z,t\} \} | = 1$ 
\item[or] a
critical word of type $\{a,b,c\}$ with $|\{x,y\} \cap \{b,c\} \} | = 1$.
\end{mylist}
\end{mylist}
\end{definition}
Note that 
our condition (iii) is slightly different  from the third condition for optimality of \cite{BCMW}; that is because we have no generator in $G$ that commutes with two other generators. 
We prefer to state our condition (iii) in this form for greater readability.
Note also that it is possible to have
two successive $\tau$-moves of type $\{a,b,c\}$ in an optimal RRS. 

We shall use the following properties of an optimal RRS.
\begin{lemma} \label{lem:properties_optRRS}
Let U be an optimal RRS of length $m$ and $1 \le i \le m$.
\begin{mylist}
\item[(i)] The word $u_i$ has no proper critical suffix of the same type as
$u_i$ that is contained within $w_i$.
\item[(ii)] Let $x$ be the first letter in $u_i$ after $\f{u_i}$ that is
not part of $\alpha_i$. Then $x$ does not have the same name as $\f{u_i}$.
\end{mylist}
\end{lemma}
\begin{proof} (i) If $u_i$ had such a suffix $w_{i2}$ with $w_i=w_{i1}w_{i2}$,
then we could define a new RRS $w_{i2},u_{i+1},\ldots,u_{m+1}$ for $w$ with
decomposition $\mu'w_{i2}w_{i+1} \cdots w_{m+1}\gamma$ where
$\mu' = \mu w_1\cdots w_{i-1}w_{i1}$, contradicting the optimality of $U$.

(ii) If $i>1$ and either $u_{i-1}$ has type $\{a,b,c\}$ or $\beta_{i-1} \ne
\emptyword$, then $x$ is the second letter of $u_i$ and is part of
$\tau(u_{i-1})$, and we see directly from the definition of $\tau(u_{i-1})$
that $x$ has a different name from $\f{u_i}$. Otherwise $x$ is in $w_i$.
We cannot have $x = \f{u_i}^{-1}$ because in that case $\widehat{u_i}$ would not
be freely reduced and so would not be critical. If $x = \f{u_i}$ then $u_i$
would have a critical suffix of the same type as $u_i$ starting at $x$,
contradicting (i). So $x$ does not have the same name as $\f{u_i}$.
\end{proof}

\begin{proposition} \label{prop:exists_optRRS}
If a word $w$ in a $(2,3,n)$-group admits an RRS then it admits an
optimal RRS.
\end{proposition}
\begin{proof}
This proof is essentially the same as the existence part of the proof of
\cite[Lemma 3.12]{BCMW} for $3$-free groups.
We can clearly restrict attention to those RRS in which the first letter of $w_1$ is as far right as possible. 

Then, given such an RRS associated to a decomposition $w_1\cdots w_{m+1}\gamma$ of $w$ for which
$\l{\tau(u_i)}=\f{w_{i+1}}$ for some $i \leq m$, we can associate the RRS $u_1,\cdots,u_i,\l{\tau(u_i)}$ with the decomposition $\mu w_1\cdots w_i\gamma'$, where\\
$\gamma'=w_{i+1}\cdots w_{m+1}\gamma$. 
And given an RRS associated with a decomposition for which $\f{\gamma}$ appears in $w_{m+1}$ we can certainly replace $w_{m+1}$ by a prefix that ends before $\f{\gamma}$. 

So we can find an RRS satisfying both conditions (i) and (ii).

We prove by induction on $m$ that we can find such an RRS satisfying all three
conditions.
For suppose we have an RRS satisfying conditions (i) and (ii), but not (iii),
for some $i$ with $1 \leq i \leq m$.
Then, as we observed above, the hypothesis of this condition could apply only
when $\alpha_i = \emptyword$ and, for the conclusion to fail, either $u_i$ would
have to be a P2G word of the same type as $u_{i-1}$, or else $u_{i-1}$ and $u_i$
would be of types $\{b,c\}$ and $\{a,b,c\}$, respectively. In both of these
cases $u_{i-1}w_i$ would be critical of the same type as $u_i$, and we could
combine $w_{i-1}$ and $w_i$ in the decomposition to contain an RRS $U'
=u_1',\ldots,u_m',u_{m+1}'$ of length $m-1$ with $u_j'=u_j$ for $j<i-1$,
$u_{i-1}' = u_{i-1}w_i$, and $u_j' = u_{j+1}$ for $i \le j \le m$. 
\end{proof}

\begin{definition}[The set $W$]
	\label{def:W}
We define $W$ to be the set of words that admit no RRS. 
\end{definition}
As we already observed,
a word that is not
freely reduced admits an RRS of length $0$ with $w_{m+1} = \emptyword$,
so all words in $W$ are freely reduced. 
Since a word admitting an RRS is transformed to a freely reducing word by its 
application, any geodesic word must be in $W$. 

We now describe a linear time procedure that, given a word $w \in W$ and 
$x \in A$, attempts to construct an optimal RRS that transforms $wx$ to a word ending in $x^{-1}x$.
The procedure may terminate with a candidate sequence and associated factorisation, or may fail
to terminate.
If the procedure fails to find a candidate sequence, or if the candidate sequence that is found fails to transform  $wx$ to a word ending in $x^{-1}x$, then
$wx$ is proved to be in $W$.

We prove correctness of this procedure in Proposition~\ref{prop:unique_optRRS},
which follows the procedure.

\begin{procedure}\label{proc:unique_optRRS}
Search a word for an optimal RRS.
\begin{mylist}
\item[{\bf Input:}] $w \in W$, $x \in A$. 
\item[{\bf Ouput:}] Either \fail\ (when $wx \in W$), or the words $u_i$ and the
decomposition $\mu w_1 \cdots w_mw_{m+1}x$ associated with the unique
optimal RRS for $wx$, together with the criticality types of the words $u_i$.
\end{mylist}

\smallskip
The procedure has two principal parts. In the first part, either \fail\ is
returned, or the data associated with a putative (unique optimal)
RRS $U$ for $wx$ is computed.

In the second part, if \fail\ was not returned, then it is checked whether
the putative RRS $U$ really is an RRS for $wx$. If so, then the data
associated with $U$ are returned, and otherwise \fail.

The first part of the procedure contains $m+1$ steps and constructs the
putative decomposition of $w$ from right to left. 

\smallskip
{\bf Step 1:} 

\begin{myproclist}
\item Identify $w_{m+1}'$ as the longest suffix of $w$ that 
commutes with $x$ and does not contain $x^{-1}$.
\item If $w_{m+1}'$ is the whole of $w$, then return \fail.
\item \emph{Check if $m=0$}: If the letter immediately to the left of the
suffix $w_{m+1}'$ within $w$ is $x^{-1}$, deduce that $m=0$, define
$u_1:=w_1 := w_{m+1}'$, return the associated
decomposition $w = \mu w_1 x$ and \stop.
\item Deduce that $m>0$, define $w_{m+1}:=w_{m+1}'$,
and identify the location of the righthand end of $w_m$ as the position within
$w$ immediately to the left of $w_{m+1}$.
Let $s,t$ be the names of $x$ and of the last letter in $w_m$.
\item	If $s=t$, then return \fail.
\item Identify the criticality type of $u_m$ as P2G of type $\{s,t\}$.
\end{myproclist}

Now for each of $i:=m,\ldots,1$, given the location of the righthand end of
$w_i$ and the criticality type of $u_i$, perform Step $m+2-i$.

\smallskip
{\bf Step $\mathbf{m+2-i}$:}

\begin{myproclist}
\item Define $\vv_i$ to be the suffix of $w$ that ends at the end of $w_i$.
\item \emph{Check if $i=1$}: if $\vv_i$ has a critical suffix of the criticality
type that has been determined for $u_i$, then deduce that $i=1$,
identify $u_1=w_1$ as the shortest critical suffix of $\vv_i$ of that type,
	and {\bf Proceed to Checking}.
\item Now $i>1$.  Select one of three different cases 1--3 below,
depending on the criticality type $\{a,b\}$, $\{b,c\}$ 
or $\{a,b,c\}$  identified for the critical word $u_i$. 
\end{myproclist}
	{\bf Note:} In each of the three cases 1--3, a letter will be identified at a 
particular position within the prefix $\vv_i$ of $w$, which we call a
\emph{distinguished letter}, and which we denote by $\disting{x}$,
where $x$ is its name.
We call the closest letters to the distinguished letter $\disting{x}$ within
$w$ that have names distinct from $x$ its \emph{differently named neighbours}
to its \emph{left} or \emph{right}, and denote them by $\leftnbr{x}$ and
$\rightnbr{x}$, respectively.

In all of the following code, if we fail to find $\disting{x}$ or
$\leftnbr{x}$ because we reach the left hand end of $w$, then we return \fail.

\smallskip
{\bf Step $\mathbf{m+2-i}$, Case 1:} the word $u_i$ is P2G-critical of type
$\{a,b\}$.

\begin{myproclist}
\item Define the \emph{distinguished letter with name $c$}, $\disting{c}$, to
be the first letter with name $c$ that is encountered within $w$ after at
least one letter with name $b$ while moving leftward from the rightmost letter
of $\vv_i$.
\item  If both $\leftnbr{c}$ and $\rightnbr{c}$ hav name $b$, 
then identify 
the location of the rightmost letter of $w_{i-1}$ 
as the position of $\leftnbr{c}$, 
and the criticality type of $u_{i-1}$ as P2G of type $\{a,b\}$.
\item 
Otherwise identify the location of the righthand end of $w_{i-1}$ 
	at $\disting{c}$, and 
the criticality type of $u_{i-1}$ as P2G of type $\{b,c\}$.
\end{myproclist}

{\bf Step $\mathbf{m+2-i}$, Case 2 or 3:} the word $u_i$ is critical of type
$\{b,c\}$ or $\{a,b,c\}$.

\begin{myproclist}
\item In Case 2,
define the \emph{distinguished letter with name $a$}, $\disting{a}$, to be the
first letter with name $a$ that is encountered within $w$ after at least one letter with name $b$ while moving leftward from the rightmost letter of $\vv_i$.
\item In Case 3,
define the \emph{distinguished letter with name $a$}, $\disting{a}$, to be the
first letter with name $a$ that is encountered within $w$ after at least one 
letter with name $b$ followed by at least one letter $c$, followed by at 
least one more letter $b$, while moving leftward from the rightmost letter of
$\vv_i$.
\item In either case, define $\leftnbr{a}$ and $\rightnbr{a}$ relative to $\disting{a}$, and locate all three of $\disting{a}$, $\leftnbr{a}$, $\rightnbr{a}$, moving leftward from the rightmost letter of $\vv_i$.

\item If both $\leftnbr{a}$ and $\rightnbr{a}$ have name $b$, 
then identify the location of the rightmost letter of $w_{i-1}$ 
	as the position of $\leftnbr{a}$,
and the criticality type of $u_{i-1}$ as P2G of type $\{b,c\}$.
\item
Otherwise identify the location of the righthand end of $w_{i-1}$ at
$\disting{a}$, the criticality type of $u_{i-1}$ as P2G of type $\{a,b\}$ or
of type $\{a,b,c\}$.
\item 	\begin{myproclist}
	\item  If $\vv_{i-1}$ has no critical P2G suffix of type $\{a,b\}$,
		then identify the criticality type of $u_{i-1}$
		as P2G of type $\{a,b\}$.
	\item Otherwise let $u'_1$ be the shortest such suffix, and define
		$u_2' := \l{\tau(\hat{u_1'})} \beta(u_1') w_i$.
\item		\begin{myproclist}
		\item If $u_2'$ is P2G-critical of type $\{b,c\}$,
		then $u'_1,u'_2$ are the first two terms $u_1,u_2$ 
		of the putative RRS $U$. Identify $i$ as $2$,
		$w_1$ as the subword $u'_1$ of $w$, the criticality type of 
		$u_{i-1}=u_1=u'_1$ as P2G of type $\{a,b\}$, and proceed to the
checking part of the procedure.
	        \item Otherwise identify the criticality type of $u_{i-1}$ 
        	as $\{a,b,c\}$.
		\end{myproclist}
	\end{myproclist}
\end{myproclist}

\medskip
{\bf Checking:} We now proceed to the second part of the procedure, where
we check whether the putative RRS $U$ that we have computed really is an
RRS for $wx$. Note that the check is unnecessary when $m=0$.

\begin{myproclist}
  \item If $m=0$, then {\bf return} the data associated with $U$.
\item Define $u_1:= w_1$.
\item	For $i:=2,\ldots,m$, do the following:
  \item \begin{myproclist}
       \item If $u_{i-1}$ is a P2G-critical word,  
	   then $u_i := \l{\tau(\widehat{u_{i-1}})} \beta(u_{i-1}) w_i$;
       \item If $u_{i-1}$ is a critical word of type $\{a,b,c\}$, then
    $u_i: = c^{\ee} b^{\kk}  \beta_{i-1} w_i$,
    using the notation of Definition~\ref{def:3gen_tau}.
    \item If $u_i$ is not critical of the specified type, then return \fail.
  \end{myproclist}
  \item Return the data associated with $U$.
\end{myproclist}
	
\end{procedure}

\begin{proposition}\label{prop:unique_optRRS}
Suppose that $w \in W$ and $x \in A$. 
The procedure described as \ref{proc:unique_optRRS} runs in linear time.
If $wx \not \in W$, then the procedure terminates with the decomposition of 
$\mu w_1 \cdots w_mw_{m+1}x$ of $w$, 
associated with the unique optimal RRS
$U=u_1,\ldots,u_m,u_{m+1}$ for the reduction of $wx$,
If $wx \in W$, then the procedure returns \fail, either because it fails
to compute a putative RRS $U$, or because $U$ fails the checking process.
\end{proposition}

\begin{proof} 
	We assume first that $wx \not \in W$,
	and that 
$\mu w_1 \cdots w_mw_{m+1}x$ is a decomposition of $w$, 
	associated with an optimal RRS $U=u_1,\ldots,u_{m+1}$;
	we know from Proposition~\ref{prop:exists_optRRS}
	that at least one optimal RRS must exist.

We shall prove that Procedure~\ref{proc:unique_optRRS}
must return the selected decomposition and RRS.
It follows from that fact that
this is true whatever optimal RRS is selected that the optimal RRS is uniquely determined.

It should be clear from Lemmas~\ref{lem:critlintest},~\ref{lem:abcrit2} and
~\ref{lem:RRScheck} that the procedure executes in linear time.

We consider the $m+1$ steps of the first part of the procedure, in the order in
which they are applied; we note that (provided that \fail\ is not returned)
the value of $m$ is determined during the 
procedure, but is not known at its start.

The purpose of  Step 1 of the procedure is to locate the subword $w_{m+1}$ within $w$ by locating the position of the final letter of the subword $w_m$,
and check whether $m=0$.  If $m>0$, then Step 1
determines the type of the critical word $u_m$.

For $i=m,\ldots,1$, the $(m+2-i)$-th step, given the location of the right hand end of $w_i$ within $w$ and the type of $u_i$, should determine first whether
$i=1$, and if so, determine the location of the left hand end of $w_i=w_1$. If $i>1$, then it should determine the location of the right hand end of $w_{i-1}$ as well as the type of the critical word $u_{i-1}$.

We use the definitions and notations of distinguished letter $\disting{x}$ and
its nearest differently named neighbours $\leftnbr{x}$ and $\rightnbr{x}$,
and of the prefixes $\vv_i$ of $w$
that were given in the description of the procedure.

\smallskip
{\bf Correctness proof for Step 1.} 

The fact that,
in any factorisation associated with an optimal RRS, 
	$w_{m+1}$ is as defined by the procedure
follows from Definitions~\ref{def:RRS} and ~\ref{def:optRRS}(ii).

If the letter immediately to the left of the suffix $w_{m+1}$ within $w$ is $x^{-1}$, 
then it follows from Definition~\ref{def:RRS} that $wx$ has an RRS
with $m=0$, which is certainly optimal (we cannot start further right).
So $m=0$ for our chosen sequence, and the procedure correctly verifies that.
Otherwise, by definition we do not have an RRS of length $0$, and the procedure 
correctly verifies that.

So now we suppose that $m>0$.
We know that the last letter of $w_m$ must be found immediately before the
first letter of $w_{m+1}$, as located by the procedure,
also that $\l{w_m}$ must have name distinct
from $x$. 

We identify the type of $u_m$ using the following argument.  By
Lemma~\ref{lem:RRSdetails}\,(i) $u_m$ cannot be critical of type $\{a,b,c\}$.
So it must be P2G-critical of type $\{a,b\}$ or $\{b,c\}$;
its two pseudo-generators are the final letters of $u_m$ and $\tau(u_m)$
(here we use Lemma~\ref{lem:P2G_taufacts} and the fact that
$\beta_m=\emptyword)$, hence the names of $\l{w_m}$ and $x$, as (correctly)
identified in Step 1 of the procedure.

\smallskip
{\bf Correctness proof for Step m+2-i.}

Next we verify the correctness of
Step $m+2-i$, for each of $i=m,\ldots,1$. 
We may assume (by induction) that the procedure enters the step having
correctly located the
right hand end of $w_i$ (which is also the right hand end of the prefix $\vv_i$ of $w$)
and the type of the critical word $u_i$.
This step should first determine
whether $i=1$; if $i=1$, it should then determine the location of the 
left hand end  of the subword $w_i=w_1$ in $w$, but if $i>1$ it should
determine the location of the
right hand end of $w_{i-1}$ (and of $\vv_{i-1}$) in $w$, together with the type of the critical word $u_{i-1}$.

\emph{Proving the correctness of the test for $i=1$.}
If $i=1$ then $w_1=u_1$ is critical and,
by Lemma~\ref{lem:properties_optRRS}\,(i), $u_1$ is the shortest critical
suffix of $\vv_i$ that has the same type as $u_i$.
Conversely, if $\vv_i$ has a critical suffix of the same type as $u_i$, then
we get an RRS for $U$ (assuming that there is one) by defining $w_1$ to
be the shortest such suffix, and this RRS is easily seen to be optimal.

\emph{So we may assume now that $i>1$.}  
We have to consider three different cases 1--3
corresponding to the three possible types $\{a,b\}$, $\{b,c\}$ or $\{a,b,c\}$ 
for the critical word $u_i$. These correspond to the three cases 1--3 listed in the procedure. 
We shall see below that the location of the right hand end of $w_{i-1}$ is dependent on the pair of criticality types of the two critical words $u_i,u_{i-1}$.
We label the three different types
$\{a,b\}$, $\{b,c\}$ or $\{a,b,c\}$ for the critical word $u_{i-1}$ as subcases
(a), (b), (c) of the three cases 1--3.
Note that, by Lemma~\ref{lem:RRSdetails}\,(iv), the pair
$u_i$, $u_{i-1}$ cannot have the pair of types $\{a,b\}$, $\{a,b,c\}$,
so we do not need to consider the subcase (c) for Case 1.

In each of the three possible criticality types for $u_i$ (Cases 1--3), we use 
the locations of a distinguished letter $\disting{x}$
and its two differently named neighbours $\leftnbr{x}$ and $\rightnbr{x}$
to determine the criticality type of $u_{i-1}$ 
and the location of the righthand end of $w_{i-1}$.
Precisely how the locations of those three letters relate to the location of
the righthand end of $w_{i-1}$ depends on the criticality type of $u_{i-1}$,
and we shall see in our case by case analysis that follows.

\smallskip
{\bf Step $\mathbf{m+2-i}$, Case 1:} the word $u_i$ is P2G-critical of type
$\{a,b\}$.

The procedure locates the distinguished letter $\disting{c}$ as the
first letter with name $c$ that we encounter within $w$ after at least one 
letter with name $b$ as we work leftward from the rightmost letter of $w_i$).
As we already explained, there are just two subcases to consider here,
1(a) and 1(b). 

\emph{Case 1(a)}: the words $u_i$ and $u_{i-1}$ are both P2G critical of type
$\{a,b\}$.

Condition (iii) of Definition~\ref{def:optRRS} (for optimality) implies that
$\alpha_i\neq \emptyword$, and hence 
$\beta_{i-1} = \emptyword$ (by Lemma~\ref{lem:RRSdetails}\,(iii))
and $u_i = \f{u_i}w_i$ (following Definition~\ref{def:RRS}). So $w_i=\s{u_i}$.

Since $\alpha_i \neq \emptyword$, by definition of $\alpha_i$ as a power of letters that commute with $\f{u_i}$,
$\f{u_i}$ must have name $a$, and $\alpha_i$ must be a power of $c$.
Then (following the definition of an $\{a,b\}$-P2G word), we 
see that $(u_i)_p$ must be an $\{a,c\}$ word. By 
Lemma~\ref{lem:properties_optRRS}\,(ii), $(u_i)_p$ can only have the form
$\f{u_i}\alpha_i$, and the letter $u_i$ immediately to the right of $(u_i)_p$
must have name $b$.

We see that $\f{u_i}=\l{\tau(u_{i-1}}$, and hence
(by Lemma~\ref{lem:P2G_taufacts}) $\l{u_{i-1}}$ ($=\l{w_{i-1}}$) has name
distinct from $a$ (within $\{a,b\}$), that is, 
it has name $b$. We need to locate that letter with name $b$, as we work leftwards from the right hand end of $w_i$ within $v$.
We know that it is the first letter to the left of the prefix
$\alpha_i$ of $w_i$, so we simply need to locate that $\alpha_i$ subword, which
is a non-trivial power of $c$.

Now, as a P2G-critical word of type $\{a,b\}$, $u_i$ must have at
its right hand end a letter with name $a$ or $b$.

If the rightmost letter of $u_i$ has name $a$, then following
Definition~\ref{def:P2G}, the suffix $(u_i)_q(u_i)_s$ 
of $u_i$ (which is all within $w_i$) must contain at least one letter with name
$b$, but any letters within that suffix that have name $c$ must be within 
$(u_i)_s$, to the right of any letters with name $b$. 
Hence the distinguished letter $\disting{c}$ 
is within the prefix $\alpha_i$ of $w_i$. The last letter of $w_{i-1}$ is a 
letter with name $b$; we see that this is $\leftnbr{c}$.

If the rightmost letter of $u_i$ has name $b$, then since $b$ does not commute 
with $c$, $\beta_i=\emptyword$, and there are no letters with name $c$ in 
either$(u_i)_s$ or $(u_i)_q$. So in this case too, the distinguished letter 
with name $c$ is within the prefix $\alpha_i$ of $w_i$, and so again the 
letter with name $b$ that is the final letter of $w_{i-1}$ is $\leftnbr{c}$.

We note (for future reference) that in both cases just considered within 1(a), 
the letter $\rightnbr{c}$ is located at the lefthand end of $(u_i)_q$, and
must have name $b$ (since by Definition~\ref{def:P2G} its name must be different from that of $\f{u_i}$, which is $a$).

To summarise,
in Case 1(a) both $\leftnbr{c}$ and $\rightnbr{c}$ have name $b$,
and the last letter of $w_{i-1}$ is $\leftnbr{c}$.
 
\emph{Case 1(b)}: the word $u_i$ is P2G-critical of type $\{a,b\}$ and
$u_{i-1}$ is P2G-critical of type $\{b,c\}$. 

We have $\f{u_i} \in \{a,b\}$, and $\l{\tau(\widehat{u_{i-1}})} \in \{b,c\}$ so,
since by Definition~\ref{def:RRS} we have $\f{u_i}=\l{\tau(\widehat{u_{i-1}})}$,
the letter $\f{u_i}$ must have name $b$.  Then, since by definition $\alpha_i$
must commute with $\f{u_i}$, we have $\alpha_i=\emptyword$.
Further, by Lemma~\ref{lem:P2G_taufacts} (i) and (ii), $\l{\widehat{u_{i-1}}}$
has name $c$, and is equal to $\l{u_{i-1}}= \l{w_{i-1}}$. 

Now, just as in Case 1(a), the rightmost letter of $u_i$ could have name $a$
or $b$. 
Since $w_i$ is a suffix of $u_i$, if the distinguished letter $\disting{c}$ is within
$w_i$ then it must be within $u_i$; but the same arguments as in the Case 1(a) tell us that if we do encounter a distinguished letter with name $c$ moving leftwards from the right hand end of $u_i$, then it would have to be within the
prefix $(u_i)_p$.
However, the prefix $(u_i)_p$ of $u_i$
can contain no letter with name $c$ since $\alpha_i=\emptyword$.
It follows that as we move leftwards from the righthand end of $\vv_i$, we will
not find $\disting{c}$ within the suffix $w_i$ of $\vv_i$.
However, we have certainly found letters with name $b$ within $w_i$, and
the rightmost letter of $w_{i-1}$, immediately to the left of $w_i$ has name $c$, so that letter must be $\disting{c}$, 
We note also that the letter to the right of this distinguished letter within 
$v_i$ (and $w$) is $\f{w_i}$, and must have name either $a$ or $b$; that letter is $\rightnbr{c}$.


Suppose that, in the current Case 1(b), $\rightnbr{c}$ has name $b$.
Then we apply Lemma~\ref{lem:properties_optRRS}\,(ii) to see that
(since $\alpha_i$ is empty) the second letter of $u_i$ cannot have the same 
name as $\f{u_i}=b$. As a P2G word of type $\{a,b\}$ whose first letter has
name $b$, the word $u_i$ cannot have second letter with name $c$, so that second
letter must have name $a$. Now $u_i=b\beta_{i-1}w_i$, and, we know that the
first letter of $w_i$ is $\rightnbr{c}$, and so has name $b$.
Hence the second letter of $u_i$ must lie in $\beta_{i-1}$, and so
$\beta_{i-1} \neq \emptyword$. Since the final letter of $u_{i-1}$ has name $c$
(by Lemma~\ref{lem:P2G_taufacts}), 
we see that
$\beta_{i-1}$ is a non-trivial power of $a$ in this case. 
We already showed that, in the current Case 1(b), $\disting{c}$ is the final
letter of $w_{i-1}$, which is also the final letter of $u_{i-1}$.
So the P2G-critical word $u_{i-1}$, which has type $\{b,c\}$, has final letter
with name $c$, and since $\beta_{i-1}$
is a non-trivial power of $a$, it follows from Definition~\ref{def:P2G} that 
$(u_{i-1})_s$ is an $\{a,c\}$ word involving both $a$ and $c$; 
so in this case $\leftnbr{c}$ must have name $a$. 

From the arguments above,
we see that in Case 1(a) both $\leftnbr{c}$ and $\rightnbr{c}$ must have name $b$,
and in Case 1(b) $\rightnbr{c}$ could have name $a$ or $b$, but if
$\rightnbr{c}$ has name $b$ then $\leftnbr{c}$ has name $a$.
Hence Cases 1(a) and 1(b) can be distinguished by the fact that 
in Case 1(a) both $\leftnbr{c}$ and $\rightnbr{c}$ have name $b$,
and
in Case 1(b) at most one of $\leftnbr{c}$ and $\rightnbr{c}$ can have name $b$.

Further, we have proved that
in Case 1(b) 
the last letter of $w_{i-1}$ is $\disting{c}$.

\smallskip
{\bf Step $\mathbf{m+2-i}$, Case 2:} the word $u_i$ is P2G critical of type
$\{b,c\}$.

In this case we need to locate the distinguished letter $\disting{a}$,
defined to be
the first letter with name $a$ that we encounter within $w$ after at least one 
letter with name $b$ as we work leftward from the rightmost letter of $w_i$.
We also need to examine
the differently named neighbours of $\disting{a}$, known as $\leftnbr{a}$ and $\rightnbr{a}$.
In this case we have three subcases to consider. 
We consider Cases 2(a) and 2(c) together.

{\it Case 2(a) and 2(c)}: the word $u_i$ is P2G-critical of type $\{b,c\}$ and 
$u_{i-1}$ is critical of type $\{a,b\}$ or $\{a,b,c\}$.

The argument for Case 1(b), but  with the roles of $a$ and $c$ swapped, 
can be applied to the case 2(a) without change, and much of that argument also applies in Case 2(c).
So we can deduce in both cases that
$\l{u_{m-1}}$ has name $a$ and $\f{u_i} = \l{\tau(u_{m-1})}$ has name $b$,
hence $\alpha_i = \emptyword$. It follows that, in both cases,
the last letter of $w_{i-1}$ is $\disting{a}$, also that
at most one of $\leftnbr{a}$ and $\rightnbr{a}$ can have name $b$.

We observe that, in Case 2(c), the word $w_{i-1}$ has a proper suffix
$(u_{i-1})_r$ that is a P2G-critical of type $\{a,b\}$.
If this situation arose in Case 2(a) then the inductive process would
stop with $i-1 = 1$ and then that suffix would be equal to $w_1$.

So we continue to read leftwards through the word $w$, starting at the position
of the rightmost letter of $w_{i-1}$. If we find that $w_{i-1}$ does not
have $\s{u}$ as a suffix, for some P2G-critical word $u$ of type $\{a,b\}$,
then we cannot be in Case 2(c), so we must be in Case 2(a).

If there is such a suffix then, if we were in Case 2(a), we would 
have found the
complete decomposition defining the putative RRS. 
We can check whether this is the case by
testing whether applying $\tau$ to this suffix results in a critical word
$u_{i-1}$ of the required type. If not, then we must be in Case 2(c),
and we continue the induction with the next value of $m+2-i$.
If so, then we have found the complete RRS if it exists at all

{\it Case 2(b)}: the words $u_i$ and $u_{i-1}$ are both P2G critical of type
$\{b,c\}$.

The argument for Case 1(a), but  with the roles of $a$ and $c$ swapped can be
applied to this case.  It follows that in this case
the last letter of $w_{i-1}$ is $\leftnbr{a}$, also that
both of $\leftnbr{a}$ and $\rightnbr{a}$ have name $b$.  of $a$ and $c$.

We see that we have matched here the three subcases (b), (c), (a) of Case 2 to 
the three routes taken in Case 2 by the procedure.

\smallskip
{\bf Step $\mathbf{m+2-i}$, Case 3:} the word $u_i$ is critical of type
$\{a,b,c\}$.

The arguments for this case are very similar to those for Case 2,
and we have copied the argument with almost no modification.
But in this case we use the slightly modified definition for the distinguished 
letter that is also used in the procedure in this case:
we define the distinguished letter $\disting{a}$ to be 
the first letter with name $a$ that is encountered within $w$ after at least one 
letter with name $b$ followed by at least one letter $c$, followed by at 
least one more letter $b$, while moving leftward from the rightmost letter of
$\vv_i$.
The differently named neighbours $\leftnbr{a}$ and $\rightnbr{a}$ are
defined relative to $\disting{a}$ as in the other cases.
Again we have to consider three cases, (a), (b) and (c), and
again we consider Subcases (a) and (c) together.

We first justify the modified definition of $\disting{a}$
that is used in all three subcases.
Reading leftwards through $\vv_i$ starting at the letter with name
$a$ at its right hand end, we first locate the first
letter with name $b$. This must lie in $(u_i)_r$ and it cannot lie in
$((u_i)_r)_s$, which is an $\{a,c\}$-word, so it must lie in $((u_i)_r)_q$.
Then (still reading leftwards) we locate the next letter with name $c$,
which could lie either in $((u_i)_r)_p$ or in $(u_i)_q$. After that we locate
the next letter with name $b$. This letter cannot lie in the
$\{a,c\}$-word ${((u_i)_r)}_p$, so it must must lie in $(u_i)_q$.
(It can be checked that our assumption $n \ge 5$ ensures that there is at least
one letter with name $b$ in $(u_i)_q$ that will be found using this process.)
Finally, we locate the next letter with name $a$, which we define to
be $\disting{a}$.

\emph{Cases 3(a) and 3(c)}: the word $u_i$ is critical of type $\{a,b,c\}$ and 
$u_{i-1}$ is critical of type $\{a,b\}$ or $\{a,b,c\}$.

The brief argument described in Cases 2(a) and 2(c) above
applies without modification in these two cases too,
giving the same conclusions, provided that the modified definitions of
$\disting{a}$, $\leftnbr{a}$ and $\rightnbr{a}$ are used.

\emph{Case 3(b)}: $u_i$ is critical of type $\{a,b,c\}$ and $u_{i-1}$ is
P2G-critical of type $\{b,c\}$.

Just as in Case 2(b), we see that
the last letter of $w_{i-1}$ is $\leftnbr{a}$, also that
both of $\leftnbr{a}$ and $\rightnbr{a}$ have name $b$.
of $a$ and $c$.

This argument completes the correctness verification of step $m+2-i$ in the
procedure in Cases 2 and 3.
Hence we have proved that if $wx \not \in W$, 
then the procedure will find the selected optimal RRS.
\end{proof}

\section{Proof of the main theorem}
\label{sec:proofs}

We recall Definition~\ref{def:W} of $W$ as the set  of words admitting no RRS.
Our main theorem will be an immediate consequence of the following result,
which specifies the rewrite system $\cR$.
\begin{theorem}
\label{thm:main_details}
Let $G$ be the Artin group with presentation
$$G= \langle a,b,c \mid aba=bab, ac=ca, {}_{n}(b,c) = {}_{n}(c,b) \rangle,$$
for $n \geq 5$, and let $W$ be as defined above.
Then the words in $W$ are precisely the words that are geodesics in $G$.

There is a process that, given an input word $w:= x_1\cdots x_n$,
runs in quadratic time to find geodesic representatives $w_0,\cdots, w_n$ in
$W$ of successive prefixes of $w$, in each case by applying an RRS of
$\tau$-moves to the word $w_{i-1}x_i$, and hence finding a geodesic
representative $w_n$ for $w$, and thereby solving the word problem in $G$.

Furthermore, if $w,w' \in W$ represent the same element of $G$ then there is
a sequence of commutations and $\tau$-moves applied to critical 2-generator
subwords that transforms $w$ to $w'$. 
\end{theorem}

This section contains the proof of Theorem~\ref{thm:main_details}.
We describe the steps of the proof now. The proof uses the results of Section~\ref{sec:RRS} together with details that are verified in
Propositions~\ref{lem:6.1}--\ref{lem:6.4} that follow. 

We need to prove that a word $w$ is in the set $W$ of words admitting no RRS
(see Definition~\ref{def:RRS}) if and only if it is geodesic.
We observed as we defined $W$ that it must contain
all geodesic words. We recall that if a word $w$ admits an RRS then it admits
an optimal RRS which, by Proposition~\ref{prop:unique_optRRS}, is unique for
words $w$ with $\p{w} \in W$.

In order to establish the first statement of Theorem~\ref{thm:main_details},
we need to show that a non-geodesic word must admit an RRS.  To do that, we 
define a binary relation $\sim$ on $A^*$ by $w \sim w'$ if and only if $w'$ can
be obtained from $w$ by carrying out a sequence of commutations and $\tau$-moves
on critical 2-generator subwords. It is clear that $\sim$ is an equivalence 
relation, and Propositions~\ref{lem:6.1}\,(i) and~\ref{lem:6.2}\,(i) below
ensure that it restricts to an equivalence relation on $W$. For $w \in W$,
we denote the equivalence class of $w$ by $[w]$ and we denote the set of all
equivalence classes within $W$ by $\Omega$.

We aim to prove that $\Omega$ corresponds bijectively to $G$, where $g \in G$
corresponds to $[w]$ for any geodesic representative $w$ of $g$. Our strategy
is to define a right action of $G$ on $\Omega$. This action will have the
property that, whenever $x \in A$ and $w,wx \in W$, we have $[w]x=[wx]$.
We shall now complete the proof of Theorem~\ref{thm:main_details}, and
provide the details of the action subsequently.

Let $w,w'\in W$ be two representatives of the same element of $G$.
Then, since all prefixes of $w$ and $w'$ lie in $W$, we see that
the elements $[w]$, $[w']$ of $\Omega$ are equal to the images of
$[\emptyword]$ under the action of the elements of $G$ represented by $w,w'$.
Since $w,w'$ represent the same element, it follows that $[w]=[w']$.
This establishes the final assertion of Theorem~\ref{thm:main_details}.
Conversely, the definition of $\sim$ ensures that any two words related by
$\sim$ must represent the same element.  Hence we have a correspondence
between the elements of $G$ and the equivalence classes of $\sim$ on $W$.

But now if $w,w' \in W$ represent the same element, with $w'$ non-geodesic 
but $w$ non-geodesic then, since the two words represent the same element we
must have $[w]=[w']$, so $w \sim w'$. But the definition of $\sim$ ensures
two equivalent words have the same length, which is a contradiction.
It follows that every word in $W$ is geodesic.

The proof of the first statement within Theorem~\ref{thm:main_details} is now
complete. So now suppose that $w=x_1\cdots x_n$ over $A$ is input.
Setting $w_0 = \emptyword$, for each $i=1,\ldots, n$ we apply
Proposition~\ref{prop:unique_optRRS} to find in time 
linear in $i$ a word $w_i \in W$ that represents $w_{i-1}x_i$. Hence in time
quadratic in $n$, we find $w_n \in W$ that represents $W$. The first statement
within Theorem~\ref{thm:main_details} ensures that $w_n$ is geodesic,
and hence we can decide whether $w=_G 1$. This completes the proof of
Theorem~\ref{thm:main_details}.

We turn now to the definition of the right action of $G$ on $\Omega$.
We start by specifying the action of an element $x \in A$ on the equivalence
class $[w]$ of a word $w \in W$. To show that this is well-defined,
it is sufficient to verify that $[w]x = [w']x$ for words $w' \in W$ for
which a single commutation or $\tau$-move on a critical 2-generator subword
transforms $w$ to $w'$.

If $wx \in W$ then, since $wx \sim w'x$, we have $w'x \in W$
and we define $[w]x := [wx] (= [w'x])$.  

If $wx \not \in W$ then $w'x \not\in W$. Then by
Proposition~\ref{prop:unique_optRRS} $wx$ has a unique optimal RRS, which
transforms $wx$ to a word of the form $vx^{-1}x$ where,
by Propositions~\ref{lem:6.1}\,(i) and~\ref{lem:6.2}\,(i),
$vx^{-1} \in W$ and hence $v \in  W$. Furthermore, $w'x$ has a unique optimal
RRS which, by Propositions~\ref{lem:6.1}\,(ii)\,(b)
and~\ref{lem:6.2}\,(ii)\,(b) is the RRS $V$ specified there (note that the
words $wx$ and $w'x$ here are the words referred to as $w$ and $w'$ in
Propositions~\ref{lem:6.1} and~\ref{lem:6.2}). Then, by
Propositions~\ref{lem:6.1}\,(ii)\,(c) and~\ref{lem:6.2}\,(ii)\,(c).
$V$ transforms $w'x$ to a word of the form $v'x^{-1}x$ with $v \sim v'$.
In this case we define $[w]x := [v] (= [w']x = [v'])$.

In either case, we have a well defined image of $[w]$ under $x$.

We need to consider the effect of applying first $x$ and then $x^{-1}$ to $[w]$.
We consider separately the two cases $wx \in W$ and $wx \not \in W$.

If $wx \in W$, then $wxx^{-1} \not \in W$, and indeed
$wxx^{-1}$ admits an RRS of length $0$ that transforms it to $w$.
Hence by definition $[wx]x^{-1} = [w]$, and so in this case
$([w]x)x^{-1} = [w]$.
If $wx \not \in W$ then, as we saw above, we have $w \sim vx^{-1}$ for
a word $v \in W$ of length $|w|-1$, and we defined $[w]x = v$.
So again we have $([w]x)x^{-1} = [vx^{-1}] = [w]$.
Hence the maps $[w]\mapsto^x [w]x$ extend to an action of the free group of rank
three on $W$, for which $[w]u = ([w]u')x$ for any word $u=u'x$ over $A$.

Our next step is to verify that, for each $w \in W$, we have
$[w]ac=[w]ca$, $[w]aba=[w]bab$ and $[w]\cdot {}_n(b,c)=[w]\cdot {}_n(c,b)$.
These results are proved as Propositions~\ref{lem:6.3},~\ref{lem:6.4} below;
those propositions imply that our action of the free group on $\Omega$ induces
to an action of $G$ as required.

We provide proofs of the Propositions~\ref{lem:6.1}--\ref{lem:6.4} below. These
correspond to Lemmas (6.1)--(6.4) of \cite{BCMW}, but in our case (where $G$ has
just 3 generators) the corresponding results are a little easier to prove.
Nonetheless our arguments were guided by those of \cite{BCMW} and we have kept
our notation as close to the notation of that article as possible, in order to
aid comparison.

\begin{proposition}\label{lem:6.1}
Suppose that the word $w \in A^*$ admits an optimal RRS
$U=u_1,\ldots,u_m,u_{m+1}$
with decomposition $w=\mu w_1 \cdots w_mw_{m+1}\gamma$.
Suppose also that $w$ has a subword $\zeta$ equal to $ac$ or $ca$, and let $w'$
be the word obtained from $w$ by replacing $\zeta$ by its reverse.
Then:
\begin{mylist}
\item[(i)] The word $w'$ admits an RRS $V$ with
decomposition $w' = \nu y_1 \cdots y_{m+1}\eta$.
\item[(ii)] Suppose that $\p{w} \in W$, so that $\gamma$ is a single letter $\xx$, and that $\zeta$ is within $\p{w}$.
Then the RRS $V$ can be chosen to ensure that:
  \begin{mylist}
  \item[(a)] $\eta=\gamma=\xx$;
	  \quad{(b)}\quad $V$ is optimal;
  \item[(c)] if $v,v'$ are the free reductions of the words 
	  ($v\xx^{-1}\xx$ and  $v'\xx^{-1}\xx$, respectively)
	  to which $w,w'$ are transformed by $U,V$,
               then $v$ can be transformed to
               $v'$ using only commutations of $a$ and $c$.
   \end{mylist}
\end{mylist}
\end{proposition}

\begin{proof} 
We observe first that, if the hypotheses of part (ii) of the proposition hold,
then the fact that $\zeta$ is within $\p{w}$ implies that $\l{w'}=\l{w}=\xx$. 
Also, we can obtain $\p{w}$ from $\p{w'}$ by a single commutation move so, if
we assume part (i) of the proposition, then $\p{w'} \not\in W$ would imply
$\p{w} \not\in W$, contrary to assumption. 
So, assuming part (i), we have $\p{w'} \in W$ and hence also
$\eta=\l{w'}=\l{w}=\gamma=\xx$.
Thus we need only prove parts (i), (ii)\,(b) and (ii)\,(c).

We write $\zeta^R$ for the reverse of $\zeta$.
We derive the RRS $V$ of part (i) and associated decomposition for $w'$ by
examining the effect that replacement of the subword $\zeta$ of $w$ by $\zeta^R$ 
has on the factorisation $\mu w_1\cdots w_{m+1}\gamma$.

We split the proof into a list of cases that depend on
where the subword $\zeta$ occurs in the decomposition of $w$.
Note that by the start  or beginning of a subword of $w$ or $w'$ we mean the location of its leftmost letter, and by the end we mean the location of its rightmost letter.

\smallskip
{\bf Case 1:} $\zeta$ is a subword of $\mu$ or of $\s{\gamma}$.

\begin{addmargin}[1em]{0em}
$w'$ has $U$ as an optimal RRS associated with a decomposition of $w'$ of the
form $\nu w_1\cdots w_{m+1}\eta$, and the results are straightforward.
\end{addmargin}

\smallskip
{\bf Case 2:} $m>0$, and $\zeta$ overlaps the right and end of $\mu$ and the
left hand end of $w_1$.

\begin{addmargin}[1em]{0em}
Recall that (by definition) $w_1=u_1$.
We define $y_1:=\f{u_1}\f{\zeta}\s{u_1} $, $v_1 := y_1$,
$y_i:= w_i, v_i := u_i$ for $i>1$, $\nu := \p{\mu}$, $\eta := \gamma$.
It is clear that $\nu y_1\cdots y_{m+1}\eta = w'$.
Then it follows from Definitions~\ref{def:P2G} and ~\ref{def:3gen_critical}
that $v_1$ is critical of the same type as $u_1$, with $\f{\zeta}=\l{\mu}$
forming part of $\alpha(v_1)$,
and by definition we have $\tau(y_1)= \f{\zeta}\tau(u_1)=\l{\mu}\tau(u_1)$.  
So $V$ is an RRS.

In order to prove that $V$ is an optimal RRS for $w'$ given the hypotheses for
part (ii), it remains
to verify those conditions of Definition~\ref{def:optRRS} for optimality that
refer to $y_1=v_1$; we just need to check conditions (iii) and (i).

Since $v_1$ has the same critical type as $u_1$ and
Definition~\ref{def:optRRS}\,(iii) holds for $U$, it must also hold for $V$.

Suppose that  Definition~\ref{def:optRRS}\,(i)
fails for $V$. In that case, Procedure~\ref{proc:unique_optRRS} ensures that 
the unique optimal RRS for $w'$ is associated with a decomposition
$\nu'y'_1y_2\cdots y_{m+1}\eta$ of $w'$, with $y'_1$ a proper suffix of $y_1$.
We see from Procedure~\ref{proc:unique_optRRS} that the criticality type
of $u_1$ is determined either by $(u_1)_r$ or, in the case when this type
could be $\{a,b\}$ or $\{a,b,c\}$, by whether $\vv_1$ has a
P2G-critical suffix of type $\{a,b\}$ that forms the first term of an
RRS for $w$. These properties are not affected by substituting $\zeta^R$ for
$\zeta$, so the criticality type of $y_1'$ is the same as that of $u_1$.
But since $u_1$ has no proper critical suffixes of the same type, this
could only happen if $|y_1'| = |u_1|-1$, and that would give
$\f{y_1'} = \f{\zeta}$, which is not a critical letter of $u_1$.
This contradiction completes the verification of the optimality of $V$.

Since application of $V$ to $w'$ yields the same result as application of $U$
to $w$, we have $v=u$ in part (ii)\,(d). 
\end{addmargin}

\smallskip
{\bf Case 3:} 
$m>0$, and the subword $\zeta$ starts at the beginning of $w_1$.

\begin{addmargin}[1em]{0em}
Let $w'_1$ be the suffix of $w_1$ of length $|w_1|-2$
(so that $w_1=\zeta w_1'$),  and define $y_1 := \f{\zeta}w_1'$,
$y_i:=w_i$ for $i>1$,
$\nu = \mu\l{\zeta}$, $\eta=\gamma$, $v_1 := y_1$ and $v_i:= u_i$ for $i>1$.
Since $w_1=u_1$ is critical, and its second letter commutes with its
first, it follows from Definitions~\ref{def:P2G} and ~\ref{def:3gen_critical}
that $v_1$ is critical of the same type as $u_1$.
Now, $\l{\zeta}=\l{\nu}$ forms part of $\alpha(u_1)$ that is not
within $\alpha(v_1)$,
so by definition we have $\tau(u_1)= \f{\zeta}\tau(v_1)$, and $V$ is an RRS.

Just as in Case 2, in order to prove that $V$ is an optimal RRS for $w'$ when
the hypotheses of (ii) hold, we need only to verify 
those conditions of Definition~\ref{def:optRRS} that refer to $y_1=v_1$;
we just need to check conditions (iii) and (i). Condition (iii) is again
clear, and the verification of Condition (i) is similar but more straightforward
to that in Case 2. And as in Case 2 we have $v=u$ in part (ii)\,(d).

We observe that we have the situation of Case 2, but with the pairs $w$, $w'$,
$U$, $V$, and $\zeta$, $\zeta^R$ interchanged.
\end{addmargin}

In each of the remaining cases, we shall see that the decompositions $U$ and
$V$ of $w$ and $w'$ start in the same position in the words $w$, $w'$. That
is, the left hand ends of $w_1$ and $v_1$ are in the same position.

This means that Condition (i) of the optimality of $V$ in part (ii)\,(c) can be
established in all
remaining cases by the following argument. If this condition
failed for $V$, then there would be another RRS $V'$ starting further to the
right in $w'$.  But then, by the optimality of $U$, applying part (i) of the
current proposition to $V'$ with the reversed substitution $\zeta^R \to \zeta$
on $w'$ would result in a decomposition $U'$ of $w$ that started further to the
left than $V'$, so it would be an instance of Case 2 of this proposition. But in
Case 2 we proved the optimality of the resulting RRS for $w'$, so by the
uniqueness of the optimal RRS we would have $U = U'$. But then our current
substitution $\zeta \to \zeta^R$ of $w$ would be an instance of Case 3 of this
proposition, which it is not.

\smallskip
{\bf Case 4:} $\zeta$ is contained within $w_i$ for some $i \le m$, but $\zeta$
does not include the first letter of $w_1$ (when $i=1$).

{\it Case 4(a)}: $\zeta$ does not contain the last letter of $w_i$.

\begin{addmargin}[1em]{0em}
Either one of the letters of $\zeta$ lies in $\alpha(u_i)$ or in $\beta(u_i)$,
or $u_i$ is of type $\{a,b,c\}$ and one of the letters of $\zeta$ is
$a^{\pm 1}$ in the subword $(u_i)_r$ (in the notation of
Definition~\ref{def:3gen_critical}) and the other is an adjacent letter
$c^{\pm 1}$ (which could be $\l{(u_i)_q}$ or $\f{\alpha(u_i)_r)}$).
We find a decomposition for $w'$ with $\nu:=\mu$, $\eta:=\gamma$,
$y_j:=w_j$ and $v_j:=u_j$ for $j \neq i$ and $y_i$ formed from $w_i$ by
replacing its subword $\zeta$ by $\zeta^R$.
That same commutation move transforms $u_i$ to a critical word $v_i$,
and since $\tau(u_i)=\tau(v_i)$, we have a corresponding RRS for $w'$ with
$v_j=u_j$ for $j\neq i$.

Conditions (ii) and (iii) of Definition~\ref{def:optRRS} for optimality of $V$
in part (ii)\,(c) follow immediately from the optimality of $U$,
and $v=u$ in part (ii)\,(d).
\end{addmargin}

\smallskip
{\it Case 4(b)}: $\zeta$ ends at the end of $w_i$. 

\begin{addmargin}[1em]{0em}
Then $\f{\zeta}$ is in $\beta(u_i)$, so $i < m$ and
$\alpha(u_{i+1})=\emptyword$. We define a decomposition for $w'$
with $\nu:=\mu$, $\eta:=\gamma$, $y_j:=w_j$ and $v_j:=u_j$ for $j \neq i,i+1$, 
but $y_i := w'_i\l{\zeta}$, where $w'_i$ is the prefix of $w_i$ of length
$|w_i|-2$, and $y_{i+1}:= \f{\zeta}w_{i+1}$.
Then $v_i := u'_i\l{\zeta}$, where $u'_i$ is the prefix of $u_i$ of length
$|u_i|-2$, is critical, because it is derived from $u_i$ by deleting a letter
of $\beta(u_i)$, and $\tau(v_i)\f{\zeta}=\tau(u_i)$.
We find that $v_{i+1} := u_{i+1}$ for $j \ne i$, and
we can check that the sequence $V=v_1,\ldots,v_{m+1}$
satisfies the conditions of Definition~\ref{def:RRS} associated with the
decomposition we have described.

Again Conditions (ii) and (iii) of the optimality of $V$ in part (ii)\,(c)
follow easily from the optimality of $U$, and $v=u$ in part (ii)\,(d).  
\end{addmargin}

\smallskip
{\bf Case 5:} $\zeta$ overlaps the end of $w_i$ and the beginning of $w_{i+1}$ 
for some $i < m$. 

\begin{addmargin}[1em]{0em}
(This reverses the transformation in Case 4(b).)
We define a decomposition of $w'$ with $\nu:=\mu$, $\eta := \gamma$,
$y_j := w_j$ and $v_j := u_j$ for $j \ne i,i+1$,
$y_i := \p{w_i}\zeta^R$ and $y_{i+1}:= \s{w_{i+1}}$.
Then $v_i := \p{u_i}\zeta^R$ is critical with $\tau(v_i) = \tau(u_i)\f{\zeta}$,
and $v_{i+1}:=u_{i+1}$.
We can check that the sequence $V=v_1,\ldots,v_{m+1}$
satisfies the conditions of Definition~\ref{def:RRS} associated with the
decomposition we have described.  Part (ii) is again straightforward.
\end{addmargin}

\smallskip
{\bf Case 6:} the end of $\zeta$ is within $w_{m+1}$. 

\begin{addmargin}[1em]{0em}
If $m>0$, then the word $w_{m+1}$ is a power of a letter that does not commute
with $\l{w_m}$, so this case cannot occur.

If $m=0$ then $w_{m+1}=w_1=\f{\gamma}^{-1} \s{w_1}$ where $\s{w_1}$ is a
(possibly empty) power of a letter.
So the leftmost letter $\f{\gamma}^{-1}$ of $w_1$ is in $\zeta$.

If this is the second letter of $\zeta$, then the decomposition of $w'$
associated with $V$ is $\p{\mu}\zeta^R \s{w_{1}} \gamma$ with
$v_1 = y_1:= \zeta^R \s{w_{1}}$.
Otherwise $\zeta$ consists of the first two letters of $w_1$ and the
associated decomposition of $w'$ is $\mu \l{\zeta} \f{\zeta} w_1' \gamma$
with $v_1=y_1 := \f{\zeta} w_1'$, where $w_1'$ is the suffix of $w_1$ of
length $|w_1|-2$.

Part (ii) is straightforward in both cases, with $v=u$ in part (ii)\,(d).
\end{addmargin}

In the remaining two cases $\zeta$ intersects $\gamma$ non-trivially, so the
hypotheses of part (ii) of the proposition do not hold.

\smallskip
{\bf Case 7:} the end of $\zeta$ is at the beginning of $\gamma$.

\begin{addmargin}[1em]{0em}
If $m>0$ then $\l{w_m}$ does not commute with $\f{\gamma}$ so $\f{\zeta}$ is
in $w_{m+1}$, and we find a decomposition of $w'$ with $\nu:= \mu$,
$y_i:= w_i$ and $v_i:=u_i$ for $i \leq m$, $y_{m+1} := \p{w_{m+1}}$
and $v_{m+1} := \p{u_{m+1}}$, and $\eta := \zeta^R\s{\gamma}$.
The case $m=0$ is also straightforward,
\end{addmargin}

\smallskip
{\bf Case 8:} the beginnings of $\zeta$ and $\gamma$ coincide.

\begin{addmargin}[1em]{0em}
(This reverses the transformation in Case 7.) If $m>0$ we find a decomposition
of $w'$ with $\nu:= \mu$, $y_i:= w_i$ and $v_i:=u_i$ for $i \leq m$,
$y_{m+1} := w_{m+1}\l{\zeta}$ and $v_{m+1} := u_{m+1}\l{\zeta}$, and
$\eta := \f{\zeta}\gamma'$, where $\gamma'$ is the suffix of $\gamma$ of
length $|\gamma|-2$.
The case $m=0$ is again straightforward,
\end{addmargin}
\end{proof}

\begin{proposition}\label{lem:6.2}
Suppose that the word $w \in A^*$ admits an optimal RRS
$U=u_1,\ldots,u_m,u_{m+1}$
with decomposition $w=\mu w_1 \cdots w_mw_{m+1}\gamma$.
Suppose also that $w$ has a
$2$-generated critical subword $\zeta$ in the generators $s$ and $t$ with
$\{s,t\}$ equal to $\{a,b\}$ or $\{b,c\}$. Let $w'$ be the word obtained from
$w$ by replacing the subword $\zeta$ by $\tau(\zeta)$.
Then:
\begin{mylist}
\item[(i)] The word $w'$ admits an RRS $V$ with
decomposition $w' = \nu y_1 \cdots y_{m+1}\eta$.
\item[(ii)] Suppose that $\p{w} \in W$, so that $\gamma$ is a single letter $\xx$, and that $\zeta$ is within $\p{w}$.
Then the RRS $V$ can be chosen to ensure that:
  \begin{mylist}
  \item[(a)] $\eta=\gamma=\xx$;
	  \quad{(b)}\quad $V$ is optimal;
  \item[(c)] if $v,v'$ are the free reductions of the words 
	  ($v\xx^{-1}\xx$ and  $v'\xx^{-1}\xx$, respectively)
	  to which $w,w'$ are transformed by $U,V$,
             then $v$ can be transformed to $v'$ using only commutations 
	       and 2-generator $\tau$-moves.
   \end{mylist}
\end{mylist}
\end{proposition}

\begin{proof}
By essentially the same argument as we used at the beginning of the proof of
Proposition~\ref{lem:6.1}, it suffices to prove parts (i), (ii)\,(b) and
(ii)\,(c).

As in the proof of Proposition~\ref{lem:6.1}, 
we split the proof into a list of cases that depend
on where the subword $\zeta$ occurs in the decomposition of $w$.
To help with readability, we shall start by listing the cases and
describing the resulting RRS $V$ of $w$, and in Cases 2--6 we shall
defer the associated proofs until later.

\smallskip
{\bf Case 1:}  $\zeta$ is within $\mu$ or $\s{\gamma}$.

\begin{addmargin}[1em]{0em}
As in Proposition~\ref{lem:6.1}, $w'$ has $U$ as an optimal RRS associated with
a decomposition of $w'$ of the form $\nu w_1\cdots w_{m+1}\eta$, and the results
are straightforward.
\end{addmargin}

\smallskip
{\bf Case 2:} The end of $\zeta$ is at the beginning of $w_1$, and either 
\begin{mylist}
\item[(a)] $m=0$; or 
\item[(b)] $w_1$ is a P2G-critical word one of whose pseudo-generators
is neither $s$ nor $t$, or $w_1$ is a critical word of type $\{a,b,c\}$ and
$\{s,t\}=\{a,b\}$; or 
\item[(c)] $\alpha(u_1) \ne \emptyword$.
\end{mylist}

\begin{addmargin}[1em]{0em}
In all of these cases $V$ has length $m+1$, and $v_1=y_1$ is defined to be the
shortest critical suffix of $\tau(\zeta)$, with $y_2:=\s{w_1}$,
$y_i:=w_{i-1}$ for $i>1$, and $v_i := w_{i-1}$ for $i \ge 1$.
\end{addmargin}

\smallskip
{\bf Case 3:} $w_1$ begins within $\zeta$, and $\zeta$ ends within $w_1$, 
and either
\begin{mylist}
\item[(a)] $\zeta$ ends at the beginning of $w_1$
and all three of the following hold:\\
(1) $m>0$; (2) $w_1$ is either P2G-critical with pseudo-generators $\{s,t\}$
or critical of type $\{a,b,c\}$ with $\{b,c\}=\{s,t\}$; and
(3) $\alpha(u_1)=\emptyword$ (i.e. we are not in Case 2); or
\item[(b)] the intersection $\zeta'$ of the two subwords $\zeta$ and $w_1$
within $w$ has length greater than one.
\end{mylist}

\begin{addmargin}[1em]{0em}
Suppose first that $\zeta$ ends at the end of $w_1$ (this can only happen
within  Case 3(b)).  Then $V$ has length $m-1$, with
$y_1 := \l{\tau(\zeta)} w_2$, $y_i:=w_{i+1}$ for $i>1$, and $v_i:=u_{i+1}$
for $i \ge 1$. This reverses the transformation in Case~2.

Otherwise (we could be in Case 3(a) or 3(b))
$w_1$ has a non-empty suffix $\pi$ that is not part of $\zeta$. In
that case $V$ has length $m$ with $y_1=v_1$ the shortest critical suffix
of $\zeta \pi$ that has the same criticality type as $u_1$,
and $y_i:=w_i$, $v_i :=u_i$ for $i>1$. 
\end{addmargin}

In each of the remaining Cases 4--8, $V$ and $U$ have the same length, with the
beginning of $w_1$ and $y_1$ in the same position in $w$ and $w'$.

\smallskip
{\bf Case 4:} $\zeta$ is a subword of $w_i$ for some $1 \le i \le m$,
but does not start at the beginning of $w_1$ (when $i=1$).

\begin{addmargin}[1em]{0em}
As we shall see in the proof below, the optimality of $U$ implies that this
can happen only when $u_i$ has type $\{a,b,c\}$ and $\zeta$ is an
$\{a,b\}$-critical subword of $(u_i)_q(u_i)_r$ in the notation of
Definition~\ref{def:3gen_critical}.
(So the intersection of $(u_i)_q$ and $\zeta$ is a possibly empty suffix
of $(u_i)_q$, and equal to a power of $b$.)
Recall from Definition~\ref{def:3gen_critical} that $\widehat{(u_i)_r}$
can be transformed to a word $b^{\ii}a^{\jj}b^{\kk}$ by using a succession
of 2-generator $\tau$-moves on $\{a,b\}$-words.

Firstly, if $\zeta = (u_i)_r$ and $\tau(\zeta) = b^{\ii} a^{\jj} b^{\kk}$,
then the component $v_i$ of $V$ has criticality
type $\{b,c\}$ with $y_i:= w_{ip}b^{\ii}$, $v_i := u_{ip}b^{\ii}$,
where $w_i=w_{ip} \zeta$ and $u_i = u_{ip} \zeta$,
$y_{i+1} := a^{\jj} b^{\kk} w_{i+1}$, $v_{i+1}:=u_{i+1}$, and
$y_j:=w_j$, $v_j:=u_j$ for $j \ne i,i+1$.

Otherwise $v_i$ has the same criticality type $\{a,b,c\}$ as $u_i$.
If the end of $\zeta$ is not at the end of $(u_i)_r$ then, after replacing
$\zeta$ by $\tau(\zeta)$ in $w_i$, the resulting word still ends with a
letter with name $a$, and we get $y_i$ and $v_i$ by replacing $\zeta$ by
$\tau(\zeta)$ in $w_i$ and $u_i$ respectively, with $y_j:=w_j$,
$v_j := u_j$ for $j \ne i$.

Finally, if $\zeta$ ends at the end of $(u_i)_r$ then, after replacing $\zeta$
in $w_i$ by $\tau(\zeta)$, the resulting word ends with a nonzero power $b^t$
of $b$, which becomes a prefix of $y_{i+1}$ in the decomposition $V$.
More precisely, writing $w_i = w_{ip} \zeta$, $u_i = u_{ip} \zeta$,
and $\tau(\zeta) = \zeta' b^t$ we have $y_i := w_{ip} \zeta'$,
$v_i := u_{ip} \zeta'$, $y_{i+1} := b^t w_i$, $v_{i+1} := u_{i+1}$,
and $y_j:=w_j$, $u_j := v_j$ for $j \ne i,i+1$.

Note that, in all of the situations arising in Case 4,
applying $V$ to $w'$ gives the same result as applying $U$ to $w$.
so $v'=v$ in part (ii) of the proposition.
\end{addmargin}

\smallskip
{\bf Case 5:} $\zeta$ starts within $w_i$ and $\zeta$ intersects
$w_{i+1}$ non-trivially for some $1 \le i < m$.

\begin{addmargin}[1em]{0em}
By using the optimality of $U$ and the assumption $n \ge 5$, we shall prove
below that there are essentially only two different possibilities that we
need consider for the types of $u_i$ and $u_{i+1}$.

In the first of these,
$u_i$ has type $\{a,b,c\}$ and $u_{i+1}$ has type $\{b,c\}$ or $\{a,b,c\}$,
and we find that $\zeta$ is an $\{a,b\}$-word with $\zeta = \zeta' b^t$ for
some suffix $\zeta'$ of $(u_i)_r$ and prefix $b^t$ of $w_{i+1}$.
Writing $w_{i} = w_{ip} \zeta'$,
$u_i = u_{ip} \zeta'$, and $w_{i+1} = b^t \pi$ for words $w_{ip}$, $u_{ip}$
and $\pi$, the word $w'$ has an RRS $V$ with $y_{i} := w_{ip} \tau(\zeta)$,
$u_i := u_{ip} \tau(\zeta)$, $y_{i+1} := \pi$, $v_{i+1} := u_{i+1}$,
and $y_j:=w_j$, $v_j := u_j$ for $j \ne i,i+1$, where
$\widehat{(v_{i})_r}$ is equivalent under $2$-generator $\tau$-moves to
$b^{\ii}a^{\jj}b^{\kk+t}$ (note that $\kk$ and $t$ must have the same signs).
This reverses the transformation that we described in the final situation
of Case 4, and applying $V$ to $w'$ gives the same result as applying
$U$ to $w$.

The second possibility is that $u_i$ has type $\{b,c\}$ and (again) $u_{i+1}$
has type $\{b,c\}$ or $\{a,b,c\}$, and $\zeta = b^{\ii}a^{\jj}b^{\kk}$
for some nonzero $\ii,\jj,\kk$,
where $w_{i} = w_{ip} b^{\ii}$, $u_i = u_{ip} b^{\ii}$ and
$w_{i+1} = a^{\jj}b^{\kk} \pi$ for some words $w_{ip}$, $u_{ip}$ and $\pi$.
Then the word $w'$ has an RRS $V$ with $y_{i} := w_{ip} \tau(\zeta)$,
$v_i := u_{ip} \tau(\zeta)$, $y_{i+1} := \pi$, $v_{i+1} := u_{i+1}$ and
$y_j:=w_j$, $v_j:=u_j$ for $j \ne i,i+1$, where $v_i$ has type $\{a,b,c\}$.
This reverses the transformation that we described in the first situation
of Case 4, and again applying $V$ to $w'$ gives the same result as applying
$U$ to $w$.

Note also that it was an application of the transformation $w \to w'$ in this
situation that resulted in Example~\ref{eg:tricky_w_in_G} and forced us to
introduce critical words of type $\{a,b,c\}$.
\end{addmargin}

\smallskip
{\bf Case 6:} $\zeta$ ends within $w_{m+1}$ but not in the first letter of
$w_1$ when $m=0$ (which was considered in Case 2). 

\begin{addmargin}[1em]{0em}
We shall prove below that this case cannot occur.
\end{addmargin}

In the final two cases $\zeta$ intersects $\gamma$ non-trivially, so we only
have to prove part (i) of the proposition.

\smallskip
{\bf Case 7:} $\zeta$ intersects both $\gamma$ and $w_mw_{m+1}$ non-trivially.

\begin{addmargin}[1em]{0em}
Since $w_{m+1}$ commutes with $\f{\gamma}$, we must have $w_{m+1}=\emptyword$.
The beginning of $\zeta$ must lie in $w_m$, or we would be in Case 5
(in which case $\zeta$ cannot intersect $\gamma$), and $u_m$ must be a P2G
word by Lemma~\ref{lem:RRSdetails}\,(i). In fact, since the two
pseudo-generators of $u_m$ are the names of $\l{w_m}$ and $\f{\gamma}$, these
must be the generators of $\zeta$.

After applying the first $m-1$ steps of the RRS $U$ to $w$, the resulting word
has the non-geodesic P2G subword $u_m\f{\gamma}$. After replacing $\zeta$ by
$\tau(\zeta)$, we still have a non-geodesic P2G subword starting in the same
place, so we can continue this RRS to get an RRS $V$ of $w'$ (of which the
right-hand end could be to the left of that of the RRS $U$).
\end{addmargin}

\smallskip
{\bf Case 8:} $\zeta$ is a subword of $\gamma$. 

\begin{addmargin}[1em]{0em}
By Case 1 we only need consider the case when 
$\zeta$ is a prefix of $\gamma$. Let $t := \f{\gamma}=\f{\zeta}$.
We apply the first $m$ moves in the RRS $U$ to obtain a word ending
in $t^{-1} w_{m+1} \gamma$. Then $t^{-1} w_{m+1} \zeta$ is a non-geodesic P2G
word, and hence so is $t^{-1} w_{m+1} \tau(\zeta)$, and we can continue our
RRS to obtain an RRS $V$ of $w'$ (of which the 
right-hand end could be to the right of that of the RRS $U$).
\end{addmargin}

\medskip
{\bf Proofs for Cases 2--6.} In these proofs we shall not repeat the
definitions of the RRS $V$ for $w'$ and we leave the reader to verify that
$V$ is indeed an RRS.

{\bf Case 2:}

\begin{addmargin}[1em]{0em}
We need to prove part (ii) of the proposition. We have
$\tau(\zeta) = \rho y_1$ for some word $\rho$ and in part (ii)\,(d)
we obtain the word $v$ from $u$ by replacing its subword $\p{\zeta}$ by
$\rho \p{\tau(y_1)}$. Since these are geodesic words representing the same
element in a $2$-generator Artin group, by Lemma~\ref{lem:2gengeo} we can
obtain one from the other using $\tau$-moves on 2-generator subwords,
which proves part (ii)\,(d).

It remains to prove the optimality of $V$ in part (ii)\,(c).
Optimality condition Definition~\ref{def:optRRS} (ii) is not affected by the
change. The optimality of $U$ together with the equalities $v_i=u_{i-1}$ for
$i>1$, implies that optimality condition (iii) holds for $i>2$.
The conditions (a), (b) and (c) of Case 2, at least one of which must hold,
now ensure that condition (iii) for optimality also holds for $i=2$.
For condition (b) is precisely the conclusion of condition (iii) for $i=2$,
while condition (a) or (c) forces the hypotheses of condition (iii) for $i=2$
to fail.

It is easy to see that $V$ satisfies optimality condition (i) when $m=0$,
so suppose that condition (i) fails with $m>0$.
Since the suffixes of $w$ and $w'$ to the right of the last letter of $y_1$
are the same, application of Procedure~\ref{proc:unique_optRRS} to $w'$ locates
the words $y_{m+2}=u_{m+1}$, $y_{m+1}=u_m$, \ldots, $y_3=u_2$ of $V'$,
in the decomposition of $V$, where $v_2$ has the same type as $u_1=w_1$.
(The argument that these types are the same is similar to that in Case 2
of Proposition~\ref{lem:6.1}: the type is determined by a proper suffix of
$v_2$ and is not affected by a change of its first letter.)

Conditions (b) and (c) of Case 2 ensure that $y_1y_2$ has no critical suffix
of the same type as $u_2$, and by considering
Procedure~\ref{proc:unique_optRRS}, we see it produces the RRS $V$ of length
$m+1$ exactly as we have defined it.  So $V$ is optimal as claimed.
\end{addmargin}


{\bf Case 3:}

\begin{addmargin}[1em]{0em}
Suppose that Case 3(b) holds. Note that the first two letters of $w_1$ have
distinct names by Lemma~\ref{lem:properties_optRRS}\,(ii).
Since the two generators $s,t$ of $\zeta$ do not commute, we have
condition (a)\,(1), that $m>0$.
Further, since $s,t$ are the first two letters of the critical word $w_1=u_1$,
it follows from Definitions~\ref{def:P2G} and~\ref{def:3gen_critical} that
$\alpha(u)_1 = \emptyword$ and so (a)\,(3) holds,
and then  (from those definitions) (a)\,(2) is immediate. 

In either case (a) or (b), if $\zeta$ ends at the end of $w_1$, then
$\beta(u_1) = \emptyword$ and $w_1$ is a $2$-generator critical word
with $\l{\tau(\zeta)} = \l{\tau(w_1)}=\l{\tau(u_1)}$.
So we find the critical word $u_2$ as a subword of $w'$,
starting at the end of the subword $\tau(\zeta)$ of $w'$, and we
define $y_1$ to be that subword.
Since $U$ is optimal, optimality for $V$ in part (ii)\,(c) could only fail if
$y_1$ had a proper critical suffix or condition (i) of
Definition~\ref{def:optRRS} failed for $i=2$. A proper critical suffix of $y_1$
would also be a proper critical suffix of $u_2$, and so within $w_2$, and hence
contradict the optimality of $U$.

Otherwise we have $w_1 = \zeta' w_{1s}$ for some non-empty suffix $\zeta'$
of $\zeta$ and a non-empty suffix $w_{1s}$ of $w_1$, and the word
$\zeta w_{1s}$ is critical of the same type as $w_1$.

Since $U$ is optimal, we only have to verify the optimality 
Definition~\ref{def:optRRS}\,(i) together with (iii) for $i=2$ in part
(ii)\,(c).
Since $w$ and $w'$ differ only to the left of their common 
suffix $w_2\cdots w_{m+1}\gamma$, Procedure~\ref{proc:unique_optRRS}
will locate the same subwords $y_{m+1},\ldots,y_2$ in $w'$ as
the subwords $x_{m+1},\ldots,x_2$ of $w$,
the same location of the right hand end of $y_1$ in $w'$ as that of $w_1$ in
$w$, and the find that the type of $v_1$ is the same as that of $u_1$.
That condition (iii) of Definition~\ref{def:optRRS} also holds, even for $i=2$
is an immediate consequence of the optimality of $U$.
For condition (i) of optimality, Procedure~\ref{proc:unique_optRRS} will
define $y_1=v_1$ exactly as in our definition of $V$, so $V$ is optimal.

In all of these cases, the words $u$ and $v$ in part (ii)\,(d) of the
proposition differ only in $2$-generator geodesic subwords that represent
the same element in the $2$-generator Artin group generated by $s$ and $t$,
and so part (ii)\,(d) holds by Lemma~\ref{lem:2gengeo}.
\end{addmargin}

In each of the remaining cases, the decompositions $U$ and
$V$ of $w$ and $w'$ start in the same position in the words $w$, $w'$. That
is, the left hand ends of $w_1$ and $v_1$ are in the same position.
By a similar argument that we used in Proposition~\ref{lem:6.1}, this implies
that condition (i) of the optimality of $V$ holds in part (ii)\,(c)
in these cases.

{\bf Case 4:}

\begin{addmargin}[1em]{0em}
In this case $\zeta$ is a subword of $u_i$. It follows from
Lemmas~\ref{lem:critsubword} and~\ref{lem:properties_optRRS}\,(i)
that $u_i$ cannot be a P2G-critical word and if $u_i$ is critical of type
$\{a,b,c\}$ then $\zeta$ cannot be a critical $\{b,c\}$-subword of that.

So $u_i$ must be critical of type $\{a,b,c\}$ with
$\zeta$ a critical $\{a,b\}$-subword of the suffix $(u_i)_q{(u_i)}_r$ of $u_i$
in the notation of Definition~\ref{def:3gen_critical}.


We see easily that optimality conditions (ii) and (iii) still hold for $V$
in part (ii)\,(c).

Otherwise, $\zeta$ ends before the end of $(u_i)_r$ and, after replacing
$\zeta$ by $\tau(\zeta)$ in $u_i$, the resulting word is still critical of
type $\{a,b,c\}$, and we get $y_i$ and $v_i$ by making this replacement.

Finally, if $\zeta$ ends at the end of $(u_i)_r$ then, after replacing $\zeta$
in $w_i$ by $\tau(\zeta)$, the resulting word ends with a nonzero power $b^t$
of $b$, which becomes a prefix of $y_{i+1}$ in the decomposition $V$.

It is easy to see that optimality condition (ii) and (iii) hold for $V$
in part (ii)\,(c).
\end{addmargin}

{\bf Case 5:}

\begin{addmargin}[1em]{0em}
\emph{Subcase 5(a)}: Suppose first that $u_i$ is critical of type
$\{a,b\}$ or $\{a,b,c\}$.

If $u_{i+1}$ is critical of type $\{b,c\}$ or $\{a,b,c\}$, then we must have
$\l{u_{i}} = \l{w_{i}} = a^{\pm 1}$.  In that case, $\zeta$ must be an
$\{a,b\}$-word and $\beta(u_{i})$ (which is a power of $c$) must be empty,
and $u_{i+1}=b^{\pm 1}w_{i+1}$ or $u_{i+1}=c^{\pm 1}b^{\kk}w_{i+1}$
when $u_{i}$ has type $\{a,b\}$ or $\{a,b,c\}$, respectively.
So the word $\alpha(u_{i+1})$ is also empty in this case, and
$\f{w_{i+1}}$ cannot have name $a$ and so it must have name $b$.
By Lemma~\ref{lem:properties_optRRS}\,(ii), this cannot happen when
$u_{i}$ has type $\{a,b\}$, so it must have type $\{a,b,c\}$, and
$\zeta = \zeta' b^t$ for some suffix
$\zeta'$ of $(u_{i})_r$ and non-empty prefix $b^t$ of $w_{i+1}$.
Checking the optimality of $V$ in part (ii)\,(c) is straightforward.

Otherwise $u_{i+1}$ is critical of type $\{a,b\}$. 
It follows from Lemma~\ref{lem:RRSdetails}\,(iv)
that this does not happen when $u_{i}$ is of
type $\{a,b,c\}$, so it must have type $\{a,b\}$.
We saw in Case 1(a) of the verification of Step $m+2-i$ within the proof of
Proposition~\ref{prop:unique_optRRS}
that $\beta(u_{i}) = \emptyword$, $\alpha(u_{i+1}) \ne \emptyword$,
$\l{w_{i}}$ has name $b$, and $\f{w_{i+1}}$ has name $c$ in this situation.
So $\zeta$ must be a $\{b,c\}$-word that intersects $w_{i}$ in a power
of $b$. The letter following $\alpha(u_{i+1})$ in $w_{i+1}$ must have name $b$
(by Lemma~\ref{lem:critsubword}).
The power of $b$ that is the subword $\zeta \cap w_{i}$
could be followed by a power of $c$ in the subword $\beta(u_{i+1})$ of
$w_{i+1}$ but, since $u_{i+1}$ is critical of type $\{a,b\}$, this
would then be followed by a letter with name $a$ and,
since we are assuming that $n \ge 5$, this is incompatible with the existence
of the critical subword $\zeta$. (Note that Example~\ref{eg:n_atleast5}
shows that this argument fails when $n=4$.)

\emph{Subcase 5(b)}: Suppose that $u_{i}$ is critical of type $\{b,c\}$.

We showed in Subcase 5(a) that $u_{i}$ and $u_{i+1}$ cannot have types
$\{a,b\}$ and $\{b,c\}$, respectively, and essentially the same argument shows
that $u_{i+1}$ cannot have type $\{a,b\}$ in Subcase 5(b).
So $u_{i+1}$ has type $\{b,c\}$ or $\{a,b,c\}$.
We saw in Cases 2(b) and 3(b) of the proof of
Proposition~\ref{prop:unique_optRRS}
that $\beta(u_{i}) = \emptyword$, $\alpha(u_{i+1}) \ne \emptyword$,
$\l{w_{i}}$ has name $b$, and $\f{w_{i+1}}$ has name $a$ in this situation.
So $\zeta$ must be an $\{a,b\}$-word that intersects $w_{i}$ in a power
of $b$. The letter following $\alpha_i$ in $w_{i+1}$
must have name $b$ and, since $n \ge 5$, this power of $b$ cannot be followed
by a letter with name $a$, so we must have $\zeta = b^{\ii}a^{\jj}b^{\kk}$
for some nonzero $\ii,\jj,\kk$.
Showing the optimality of $V$ in part (ii)\,(c) is straightforward.
\end{addmargin}

{\bf Case 6:}

\begin{addmargin}[1em]{0em}
When $m=0$ and $|w_1| \ge 2$ the first two letters of $w_{m+1}$ are distinct
and commute, so this case cannot occur with $m=0$.
So $m > 0$ and $w_{m+1}$ is a power of $a$ or $c$, and $\zeta$ must intersects
$w_m$ non-trivially.
By Lemma~\ref{lem:RRSdetails}\,(i) $u_m$ cannot have type $\{a,b,c\}$
in an optimal RRS, so it has type $\{a,b\}$ or $\{b,c\}$ with
$\beta_m=\emptyword$.  Since $w_{m+1}$ commutes with one of the
pseudo-generators of $u_m$, this case cannot occur.
\end{addmargin}
\end{proof}

\begin{proposition}\label{lem:6.3}
We have $[wac] = [wca]$ for all $w \in W$.
\end{proposition}
\begin{proof}
Suppose first that $wa \in W$. Then $wac \not\in W$ if and only if $wac$ admits
an RRS the application of which replaces the final $a$ of $wa$ by $c^{-1}$,
in which case $w$ is equivalent to a word of form $w' c^{-1}$.
But that is the case if and only if $wc \not \in W$, in which case $[wac]=[wca]
=[w'a]$. Otherwise $wac$ and $wc$ are both in $W$, in which case so is $wca$
by Proposition~\ref{lem:6.1} and, since $wac$ and $wca$ are equivalent words
in $W$, we again have $[wac]=[wca]$. The proof when $wc \in W$ is analogous.

It remains to deal with the case when neither $wa$ no $wc$ is in $W$.
Then $w$ is equivalent to a word of form $w' a^{-1}$. Since
$w' a^{-1} c \not\in W$, reasoning as in the previous paragraph, we see that
$w'$ is equivalent to a word of form $w'' c^{-1}$, and then $[wac]=[wca]=w''$.
\end{proof}

\begin{proposition}\label{lem:6.4}
We have $[waba] = [wbab]$  and $[w{}\cdot{}_n(b,c)] = [w{}\cdot{}_n(c,b)]$
for all $w \in W$.
\end{proposition}
\begin{proof}
We prove that $[waba] = [wbab]$. The proof that
$[w\cdot{}_n(b,c)] = [w\cdot{}_n(c,b)]$ is analogous.
We start by replacing $w$ by
an equivalent word with suffix a longest possible $\{a,b\}$-word. So
$w = w' u$, where $u$ is an $\{a,b\}$-word, and $w'$ is not equivalent to
any word ending in a letter with name $a$ or $b$.

We claim that all reductions resulting from applications of RRS's when
computing $[waba]$ result from RRS's of length $1$ consisting of a single
$2$-generator critical $\{a,b\}$-word
that does not intersect the prefix $w'$ of
the word, and similarly for $[wbab]$. Lemma~\ref{lem:2gengeo} will then imply
that $[waba]$ and $[wbab]$ are both equal to $[w'v]$, where $v$ is a geodesic
representative of $uaba =_G ubab$ in the $2$-generator Artin group
$\langle a,b \mid aba=bab \rangle$.

To prove the claim, suppose not. Then there are words $w'v$ and $w'vx$ with
$v$ an $\{a,b\}$-word, $x \in \{a,a^{-1},b,b^{-1}\}$, $w'v \in W$,
and $w'vx \not\in W$, but $vx \in W$. Let $U$ be the optimal RRS of $w'vx$,
which replaces $w'v$ by a word ending in $x^{-1}$. If $v$ is empty, then $w'$
is equivalent to a word ending in $x^{-1}$, contradicting the choice of $w$.
So $v$ is non-empty and, in the decomposition associated with $U$, $w_{m+1}$
is empty, and $\l{v} = \l{w_m}$ has name different from $x$, so $u_m$
must be critical of type $\{a,b\}$. Furthermore since $vx \in W$ we
must have $m > 1$. We cannot have $\l{w_{m-1}}$ in $v$ because that would mean
$u_{m-1}$ was critical of type $\{a,b\}$ or $\{a,b,c\}$, contradicting
condition (iii) of the optimality of $U$. So $\l{w_{m-1}}$ is in $w'$. In fact
either $\l{w_{m-1}} = \l{w'}$ or $w'$ has a suffix consisting of $\l{w_{m-1}}$
followed by a power of $c$, where $\l{\tau(u_{m-1})} = a^{\pm 1}$. In either
case $w'$ is equivalent to a word ending in a letter in $\{a,a^{-1},b,b^{-1}\}$,
contradicting the choice of $w$.
\end{proof}

\end{document}